\documentclass[11pt]{amsart}
\usepackage{amscd}
\usepackage{amsfonts}
\usepackage{amssymb,latexsym}
\usepackage{enumerate}
\usepackage{amsmath}
\usepackage{mathrsfs}
\usepackage{amsthm}
\usepackage{hyperref}
\usepackage[square, comma, sort&compress, numbers]{natbib}
\usepackage{bm}
\usepackage{esint}
\usepackage{fancyhdr}
\usepackage{lastpage}
\usepackage{layout}
\usepackage{ulem}
\usepackage{extarrows}
\usepackage{geometry}
\newcommand{\ol}{\overline}

\newcommand{\pa}{\partial}

\DeclareMathOperator\tr{tr}
\DeclareMathOperator\id{id}
\DeclareMathOperator\rank{rank}

\DeclareMathOperator\dist{dist}
\DeclareMathOperator\vol{vol}
\DeclareMathOperator\dvol{dvol}
\DeclareMathOperator\End{End}
\DeclareMathOperator\loc{loc}

\DeclareMathOperator\HE{HE}
\DeclareMathOperator\AHE{AHE}
\begin{document}
\newcounter{remark}
\newcounter{theor}
\setcounter{remark}{0}
\setcounter{theor}{1}
\newtheorem{theorem}{Theorem}[section]
\newtheorem{lemma}[theorem]{Lemma}
\newtheorem{corollary}[theorem]{Corollary}
\newtheorem{corollarys}[theorem]{Corollary}
\newtheorem{proposition}[theorem]{Proposition}
\newtheorem{claim}[theorem]{Claim}
\newtheorem{question}[theorem]{question}
\newtheorem{defn}[theorem]{Definition}
\newtheorem{examp}[theorem]{Example}
\newtheorem{assumption}{Assumption}[section]
\newtheorem{case}{Case}
\newtheorem*{maintheorem}{Main Theorem}
\newtheorem{remark}[theorem]{Remark}
\newtheorem*{theorem1}{Theorem}
\numberwithin{equation}{section}

\title[Hermitian-Einstein equations on noncompact manifolds]
{Hermitian-Einstein equations on \\ noncompact manifolds}
\author{Di Wu}
\address{Di Wu, School of Mathematics and Statistics, Nanjing University of Science and Technology, Nanjing 210094, People's Republic of China}
\email{wudi@njust.edu.cn}
\author{Xi Zhang}
\address{Xi Zhang, School of Mathematics and Statistics, Nanjing University of Science and Technology, Nanjing 210094, People's Republic of China}
\email{mathzx@njust.edu.cn}
\thanks{The research is supported by the National Key R\&D Program of China 2020YFA0713100. Both authors are partially supported by the National Natural Science Foundation of China No.12141104 and No.12431004. The first author is also supported by the project funded by China Postdoctoral Science Foundation 2023M731699, the Jiangsu Funding Program for Excellent Postdoctoral Talent 2022ZB282, and the Natural Science Foundation of Jiangsu Province No.BK20230902.}
\maketitle
\begin{abstract}
	This paper first investigates solvability of Hermitian-Einstein equation on a Hermitian holomorphic vector bundle on the complement of an arbitrary closed subset in a compact Hermitian manifold. The uniqueness of Hermitian-Einstein metrics on a Zariski open subset in a compact K\"{a}hler manifold was only figured out by Takuro Mochizuki recently, for this model the second part of this paper gives an affirmative answer to a question proposed by Takuro Mochizuki and it leads to an alternative approach to the unique issue. We also prove stability from solvability of Hermitian-Einstein equation, which together with the classical existence result of Carlos Simpson in particular establish a Kobayashi-Hitchin bijective correspondence. The argument is also effective in more general settings, including basic models of Takuro Mochizuki, as well as non-K\"{a}hler and semi-stable contexts.
\end{abstract}
\section{Introduction}
\par Let $(E,\ol\pa_E)$ be a holomorphic vector bundle on a compact K\"{a}hler manifold $(N,J,g)$, the Hermitian-Einstein (abbreviated as $\HE$, also called Hermitian-Yang-Mills because of its relation to Yang-Mills theory) equation reads as
\begin{equation}\begin{split}\label{HE}
		\Lambda F_H^\perp\triangleq\Lambda F_H-\frac{\tr\Lambda F_H}{\rank(E)}\id_E=0,
\end{split}\end{equation}
where $\Lambda$ is the operator obtained as the adjoint of the multiplication of the associated fundamental form $\omega_g=g(J(\cdot),\cdot)$, and $F_H$ denotes the Chern curvature of a Hermitian metric $H$ on $(E,\ol\pa_E)$. The $\HE$ equation is one of the most important nonlinear partial differential equations in complex geometry and solving such an equation provides an effective approach to construct Yang-Mills connections on holomorphic vector bundles.  A Hermitian metric satisfying the $\HE$ equation will be called a $\HE$ metric and a holomorphic vector bundle will be called indecomposable if it is not a direct sum of serval holomorphic vector bundles. Thanks to a series of works
including those of Narasimhan-Seshadri \cite{NS1965}, Kobayashi \cite{Ko1982}, L\"{u}bke \cite{Lu1983}, Metha-Ramanathan \cite{MR1984}, Donaldson \cite{Do1985,Do1987} and Uhlenbeck-Yau \cite{UY1986}, the celebrated Kobayashi-Hitchin correspondence (also called Hitchin-Kobayashi correspondence or Donaldson-Uhlenbeck-Yau correspondence) states that for an indecomposable holomorphic vector bundle, the existence and uniqueness of $\HE$ metrics is equivalent to the Mumford-Takemoto stability. It is a landmark in complex geometry and has been generalized in many different directions, the large majority of which themselves became very important results in complex geometry. 
\par Furthermore, the semi-stable Kobayashi-Hitchin correspondence \cite{Ja2014,Ko1987,LZ2015} indicates the equivalence between the existence of an approximate Hermitian-Einstein (abbreviated as $\AHE$) structure and the Mumford-Takemoto semi-stability, where an $\AHE$ structure is a family of Hermitian metrics $\{H_\epsilon\}_{\epsilon>0}$ such that the $\HE$ equation can be solved in the sense of
\begin{equation}\begin{split}\label{HE2}
		\sup\limits_{N}|\Lambda F_{H_\epsilon}^\perp|_{H_\epsilon}\xrightarrow{\epsilon\rightarrow0}0.
\end{split}\end{equation}
In fact, to define stability and semi-stability, it is enough to pick a Gauduchon metric (that is, $\pa\ol\pa\omega_g^{\dim_{\mathbb{C}}N-1}=0$) as the reference metric. By a result of Gauduchon \cite{Ga1977}, there exists on any compact complex manifold a Gauduchon metric in a fixed conformal class. Therefore, the stability and semi-stability also make sense on compact Hermitian manifolds and the above two mentioned results on $\HE$ equations also hold on non-K\"{a}hler manifolds \cite{Bu1988,LY1987,NZ2018} indeed. For more related descriptions on this problem, the reader may also consult \cite{LT1995,LT2006}.
\par It is a natural to ask under what assumptions $\HE$ equations can be solved on noncompact K\"{a}hler manifolds and this question had been systematically investigated in various contexts, see \cite{Ba1993,LW1999,MS2010,Mo2020,Ni2002,NR2001,Si1988,WZpreprint,ZZZ2021,Zh2022} and references therein. From both the theoretic point of view and that of applications, it is important and highly desirable to establish a general existence theory for $\HE$ equations on noncompact spaces with a few technical assumptions as possible, so that it covers a wide range of applications in different areas. 
\par A first purpose of this paper is to study the above issue on a new but natural class of noncompact manifolds. We first have the following result.
\begin{theorem}\label{thm1}
	Let $(M,J)$ be a compact complex manifold and $\Sigma\subset M$ be an arbitrary closed subset. Suppose $g$ is a Hermitian metric on $M\setminus\Sigma$ equivalent to a Riemannian metric on $M$ and $\sqrt{-1}\pa\ol\pa\omega_g^{\dim_{\mathbb{C}}M-1}\geq0$. If $(E,\ol\pa_E,K)$ is a stable Hermitian holomorphic vector bundle on $M\setminus\Sigma$ with $|\Lambda F_{K}^\perp|\in L^\infty(M\setminus\Sigma)$, there exists a $\HE$ metric compatible with $K$.
\end{theorem}
The general strategy in proving Theorem \ref{thm1} is not so involved. In fact, one may first solve Dirichlet problems of $\HE$ equations on an exhaustion series of compact submanifolds with boundaries. Secondly, once the zeroth estimate of a sequence of solutions to Dirichlet problems is established, the evolved metrics will converge to the desired $\HE$ metric. Alternatively, the strategy in proving Theorem \ref{thm2} below also applies to Theorem \ref{thm1}, see Remark \ref{stableremark}.
\par We are further interested in the semi-stable context and prove
\begin{theorem}\label{thm2}
	Let $(M,J)$ be a compact complex manifold and $\Sigma\subset M$ be an arbitrary closed subset. Suppose $g$ is a Hermitian metric on $M\setminus\Sigma$ equivalent to a Riemannian metric on $M$ and $\sqrt{-1}\pa\ol\pa\omega_g^{\dim_{\mathbb{C}}M-1}\geq0$. If $(E,\ol\pa_E,K)$ is a semi-stable Hermitian holomorphic vector bundle on $M\setminus\Sigma$ with $|\Lambda F_{K}^\perp|\in L^\infty(M\setminus\Sigma)$, there exists an $\AHE$ structure compatible with $K$.
\end{theorem}
In Theorems \ref{thm1}, \ref{thm2}, we require neither K\"{a}hler condition nor Gauduchon condition on $\omega_g$. We should mention that if $\Sigma\subset M$ is Zariski closed and $g$ is the restriction of a K\"{a}hler metric on $M$ (we call it the model of Simpson since it satisfies the key assumptions in \cite{Si1988}),  Theorem \ref{thm1} was proved in \cite{Si1988} and Theorem \ref{thm2} was obtained in \cite{ZZZ2021}. Actually, the authors in \cite{ZZZ2021} studied the case of noncompact Gauduchon manifolds satisfying Simpson's type assumptions (see \cite[Section 2]{Si1988}) and $|d\omega^{\dim_{\mathbb{C}}N-1}_g|\in L^2(N)$. Different from the standard method of exhaustion in stable case, the strategy we adopt in semi-stable case is to extend the method of continuity by Uhlenbeck-Yau \cite{UY1986} to noncompact case, meaning to solve a family of perturbed $\HE$ equations $(\ref{epsilonequation})$ on the noncompact space $M\setminus\Sigma$. The solvability of them was established on compact Gauduchon manifolds in \cite{LY1987,UY1986}, on noncompact Gauduchon manifolds satisfying Simpson's type assumptions and  $|d\omega^{\dim_{\mathbb{C}}N-1}_g|\in L^2(N)$ in \cite{ZZZ2021} with the help of solving the following Poisson equation:
\begin{equation}\begin{split}
		\sqrt{-1}\Lambda\pa\ol\pa f=-\frac{1}{\rank(E)}\tr\left(\sqrt{-1}\Lambda F_K-\frac{\int_N\sqrt{-1}\Lambda F_K\dvol_g}{\rank(E)\vol(N)}\id_E\right).
\end{split}\end{equation}
However, the solvability of such a equation is not clearly since there is no enough information on both $\Sigma$ and $\omega_g$. To get around this obstacle, the observation is a carefully new choice of a perturbed heat flow $(\ref{flow})$ and the advantage is that, this flow preserves the property $\det H=\det K$. Finally, the existence of an $\AHE$ structure will be proved by examining the limit of $\{H_\epsilon\}_{\epsilon>0}$ when $\epsilon$ goes to zero. In this step, one should note that again the generality of $\Sigma$ and $\omega_g$ may not imply suitable Stokes lemma on $M\setminus\Sigma$ which is a pivotal point in establishing global integral estimates \cite{Si1988,ZZZ2021}, our argument removes this technical constraint.
\par Below we consider the uniqueness problem of $\HE$ metrics on a stable holomorphic vector bundle. In compact case, it is not hard to know that $\HE$ metrics are unique up to constant multiples, see \cite{Ko1987,LT1995} for example. However, the noncompact issue becomes more subtle. Even for the case that $\Sigma\subset M$ is a Zariski closed subset, the uniqueness of $\HE$ metrics is unknown until the recent work \cite{Mo2020} of Mochizuki, while the existence of $\HE$ metrics was established by Simpson \cite{Si1988} more than thirty years ago. Explicitly, we have
\begin{theorem}[Mochizuki, Remark 2.13 in \cite{Mo2020}]\label{thm3}
Let $(M,J,g)$ be a compact K\"{a}hler manifold and $\Sigma\subset M$ be a Zariski closed subset. Suppose $(E,\ol\pa_E,K)$ is a stable Hermitian holomorphic vector bundle on $M\setminus\Sigma$ with $|\Lambda F_K|\in L^\infty(M\setminus\Sigma)$ and $H_1, H_2$ are two $\HE$ metrics compatible with $K$, then we have $H_1=H_2$.
\end{theorem}
In the same paper, Mochizuki also pointed out (see the paragraph above \cite[Remark 2.13]{Mo2020}) that the key point is, it is not clear whether the stability of $(E,\ol\pa_E,K)$ implies the stability of $(E,\ol\pa_E,H)$ for a $\HE$ metric $H$ and we need additional work. This nuisance does not appear in the case when the base space is compact since the stability condition on compact manifolds does not depend on the choice of metrics on a holomorphic vector bundle. In this paper, we answer Mochizuki's question affirmatively for the model of Simpson as follows.
\begin{theorem}\label{thm4}
Let $(M,J,g)$ be a compact K\"{a}hler manifold and $\Sigma\subset M$ be a Zariski closed subset. Suppose $(E,\ol\pa_E,K)$ is a stable Hermitian holomorphic vector bundle on $M\setminus\Sigma$ with $|\Lambda F_K|\in L^\infty(M\setminus\Sigma)$ and $H$ is a $\HE$ metric compatible with $K$, then $(E,\ol\pa_E,H)$ is stable.
\end{theorem}
As a direct consequence of Theorem \ref{thm4}, we shall present an alternative approach to the uniqueness issue of $\HE$ metrics on noncompact manifolds in Section 4.
\par Next we focus on the necessary condition for the existence of $\HE$ metrics on the model of Simpson. It is easy to see that the existence of a $\HE$ metric $H$ on an indecomposable holomorphic vector bundle $(E,\ol\pa_E)$ must imply the stability of $(E,\ol\pa_E,H)$. However, similar as Mochizuki's question mentioned above, the stability with respect to the background data $(E,\ol\pa_E,K)$ is generally unknown, hence a Kobayashi-Hitchin bijective correspondence between stability and existence of $\HE$ metrics remains open. In present work, we are able to prove stability (resp. semi-stability) of $(E,\ol\pa_E,K)$ from  stability (resp. semi-stability) of $(E,\ol\pa_E,H)$, therefore establishing the opposite direction of Theorems \ref{thm1}, \ref{thm2} as follows.
\begin{theorem}\label{thm5}
Let $(M,J,g)$ be a compact K\"{a}hler manifold and $\Sigma\subset M$ be a Zariski closed subset. Suppose $(E,\ol\pa_E,K)$ is an indecomposable Hermitian holomorphic vector bundle on $M\setminus\Sigma$ with $|\Lambda F_K|\in L^\infty(M\setminus\Sigma)$, then we have
\begin{enumerate}
	\item If there exists a $\HE$ metric compatible with $K$, then $(E,\ol\pa_E,K)$ is stable.
	\item If there exists an $\AHE$ structure compatible with $K$, then $(E,\ol\pa_E,K)$ is semi-stable.
	\end{enumerate}
\end{theorem}
Gathering Simpson's existence result in \cite{Si1988} with the stability result in Theorem \ref{thm5}, we conclude a Kobayashi-Hitchin bijective correspondence for the model of Simpson. 
\begin{corollary}[Kobayashi-Hitchin bijective correspondence]
Let $(M,J,g)$ be a compact K\"{a}hler manifold and $\Sigma\subset M$ be a Zariski closed subset. Suppose $(E,\ol\pa_E,K)$ is an indecomposable Hermitian holomorphic vector bundle on $M\setminus\Sigma$ with $|\Lambda F_K|\in L^\infty(M\setminus\Sigma)$, then 
\begin{center}
	$(E,\ol\pa_E,K)$ is stable $\Leftrightarrow$ it admits a unique $\HE$ metric compatible with $K$.
\end{center}
\end{corollary}
By Theorems \ref{thm2}, \ref{thm5}, we also obtain
\begin{corollary}[Semi-stable Kobayashi-Hitchin bijective correspondence]
	Let $(M,J,g)$ be a compact K\"{a}hler manifold and $\Sigma\subset M$ be a Zariski closed subset. Suppose $(E,\ol\pa_E,K)$ is a Hermitian holomorphic vector bundle on $M\setminus\Sigma$ with $|\Lambda F_K|\in L^\infty(M\setminus\Sigma)$, then 
	\begin{center}
		$(E,\ol\pa_E,K)$ is semi-stable $\Leftrightarrow$ it admits an $\AHE$ structure compatible with $K$. 
	\end{center}
\end{corollary}
It may be mentioned that Mochizuki's discussion in \cite{Mo2020} relies on certain crucial properties of the Donaldson functional on the space of Hermitian metrics and a priori estimate of $\HE$ equation due to Simpson \cite{Si1988}, as well as the celebrated regularity result of weakly holomorphic subbundles due to Uhlenbeck-Yau \cite{UY1986}. In a different spirit, without the recourse to above deep results, our approach which is based on Theorem \ref{thm4}, seems simpler and has a merit that the proof can be readily applied to the case where the base space is $\mathbb{C}$ or $\mathbb{C}\times M$ for a compact K\"{a}hler manifold $M$, and also to the non-K\"{a}hler context. In fact, we also build
\begin{enumerate}
	\item Uniqueness of $\HE$ metrics on $\mathbb{C}$, $\mathbb{C}\times M$ for a compact K\"{a}hler manifold $M$, and noncompact Gauduchon manifolds, see Theorems \ref{C1}, \ref{CM1}, \ref{thm71}. 
	\item Kobayashi-Hitchin bijective correspondence on $\mathbb{C}$ and noncompact Gauduchon manifolds, see Theoremss \ref{C2}, \ref{thm72}. We should note that the existence part on $\mathbb{C}$ was previously given by Mochizuki \cite{Mo2020}.
	\item Semi-stable Kobayashi-Hitchin bijective correspondence on noncompact Gauduchon manifolds, see Theorem \ref{thm73}.
\end{enumerate}
\par We should remark that $\mathbb{C}$ and $\mathbb{C}\times M$ for a compact K\"{a}hler manifold $M$ are basic models satisfying Mochizuki's assumption (see \cite[Assumption 2.1]{Mo2020}). Compared with that of Simpson, the model of Mochizuki allows the possibility of infinite volume. But due to the infinite volume, one needs additional work to investigate the semi-stable Kobayashi-Hitchin bijective correspondence on $\mathbb{C}$ and it will be addressed elsewhere, see \cite{WZpreprint}.
\par
\textbf{Organization of this paper.} In Section 2, we give some preliminaries that will be used in later proofs. In Section 3, we study solvability of $\HE$ equations on noncompact manifolds and prove Theorems \ref{thm1}, \ref{thm2}. In Section 4, we investigate uniqueness and stabilities of $\HE$ bundles on noncompact manifolds by giving an alternative proof of Theorem \ref{thm3} and proving Theorems \ref{thm4}, \ref{thm5}. Related results on $\mathbb{C}$ and non-K\"{a}hler manifolds are also presented.
\section{Preliminaries}
Let $(E,\ol\pa_E)$ be holomorphic vector bundle on a Hermitian manifold $(N,J,g)$. For any two Hermitian metrics $K$, $H$ on $E$, we always denote by $h=K^{-1}H\in C^\infty(N,\End(E))$ the section defined by $H(\cdot,\cdot)=K(h(\cdot),\cdot)$. Throughout this paper, unless otherwise mentioned, inner products and norms are always taken with respect to $K$ and $g$.
\begin{defn}
	We say $H$ is compatible with $K$ if
	\begin{equation}\begin{split}
			\det K=\det H,\ |h|\in L^\infty(N),\ |\ol\pa_Eh|\in L^2(N).
	\end{split}\end{equation}
\end{defn}
Let us recall the definition of stability in \cite{Si1988}. We define
\begin{equation}\begin{split}
		\deg(E,K)=\int_N\tr\sqrt{-1}\Lambda F_{K}\dvol_g,
\end{split}\end{equation}
where $F_{K}=D_K^2$ and $D_K$ is the Chern connection. Given a saturated sub-sheaf $\mathcal{S}\subset\mathcal{O}_N(E)$, there exists a closed complex analytic subset $Z(\mathcal{S})$ of complex codimension at least two such that $\mathcal{S}|_{N\setminus Z(\mathcal{S})}$ becomes a locally free $\mathcal{O}_N$-module inheriting the holomorphic structure and metric from $E$, and therefore
$\deg(\mathcal{S},K)$ can be naturally defined by integrating outside $Z(\mathcal{S})$.
\begin{defn}\label{def}
	We say $(E,\ol\pa_E,K)$ is (semi-)stable if
	\begin{equation}\begin{split}
	\frac{\deg(\mathcal{S},K)}{\rank(\mathcal{S})}(\leq)<\frac{\deg(E,K)}{\rank(E)},
	\end{split}\end{equation}
for any nontrivial saturated sub-sheaf $\mathcal{S}$.
\end{defn}
\begin{remark}
	In fact, we have the Chern-Weil formula
	\begin{equation}\begin{split}\label{ChernWeilformula}
			\deg(\mathcal{S},K)=\int_{N\setminus Z(\mathcal{S})}(\tr(\sqrt{-1}\pi_K\circ\Lambda F_{K})-|\ol\pa_E\pi_K|^2)\dvol_g,
	\end{split}\end{equation}
	where $\pi_K$ denotes the projection onto $\mathcal{S}|_{N\setminus Z(\mathcal{S})}$ with respect to $K$.
\end{remark}
It is not hard to see
\begin{equation}\begin{split}
		D_H=D_K+h^{-1}\pa_{K}h,\ F_H=F_K+\ol\pa_E(h^{-1}\pa_Kh),
\end{split}\end{equation}
\begin{equation}\begin{split}
		\sqrt{-1}\Lambda\pa\ol\pa\tr h
		=(\sqrt{-1}\Lambda F_{K}-\sqrt{-1}\Lambda F_{H},h)
		+|\ol\pa_E h\circ h^{-\frac{1}{2}}|^2,
\end{split}\end{equation}
where $\pa_K$ is the $(1,0)$-component of $D_K$. By \cite[Section 7.4]{LT1995} and \cite[Section 4]{Si1988}, for $x\in N$ and a $K$-unitary basis $\{e_a\}_{a=1}^{\rank(E)}$, we set
\begin{equation}\begin{split}
		\log h(x)=\sum\limits_{a=1}^{\rank(E)}\lambda_ae_a\otimes e^a,\ \ol\pa_E\log h(x)=\sum\limits_{a,b=1}^{\rank(E)}A_a^be^a\otimes e_b,
\end{split}\end{equation}
\begin{equation}\begin{split}
		\Theta[\log h](\ol\pa_E\log h)(x)=\sum\limits_{a,b=1}^{\rank(E)}\Theta(\lambda_a,\lambda_b)A_a^be^a\otimes e_b,
\end{split}\end{equation}
where $\Theta(x,y)=\frac{e^{y-x}-1}{y-x}$ for $x\neq y$ and $\Theta(x,y)=1$ for $x=y$. Then we have
\begin{equation}\begin{split}\label{keyequality}
		\sqrt{-1}\Lambda\pa\ol\pa|\log h|^2
		&=-2\sqrt{-1}\Lambda\ol\pa<\pa_K\log h,\log h>
		\\&=-2\sqrt{-1}\Lambda\ol\pa<h^{-1}\pa_Kh,\log h>
		\\&=-2\sqrt{-1}\Lambda<\ol\pa_E(h^{-1}\pa_Kh),\log h>
		+2\sqrt{-1}\Lambda<h^{-1}\pa_Kh,\pa_K\log h>
		\\&=2<\sqrt{-1}\Lambda(F_{K}-F_{H}),\log h>
		+2<\Theta[\log h](\ol\pa_E\log h),\ol\pa_E\log h>,
\end{split}\end{equation}
where we have used 
\begin{equation}\begin{split}
		\tr(h^{-1}\pa_{K}h\circ\log h)=\tr(\log h\circ\pa_{K}\log h),
\end{split}\end{equation}
and the fact (see \cite[$(3.12)$]{NZ2018})
\begin{equation}\begin{split}
\sqrt{-1}\Lambda<h^{-1}\pa_{K}h,\pa_K\log h>=<\Theta[\log h](\ol\pa_E\log h),\ol\pa_E\log h>.
\end{split}\end{equation}
\par
We denote by $\nabla^H$ the connection induced by $D_{H}$ and the Chern connection $\nabla_g$ on $TN$, by $\nabla$ the connection induced by $D_{K}$ and $\nabla_g$. In the local holomorphic coordinate $(z^1,...,z^{\dim_{\mathbb{C}}N})$ of $N$ and the holomorphic frame $\{e_1,...,e_{\rank E}\}$ of $E$, we have
\begin{equation}\begin{split}\label{AA}
		\sqrt{-1}\Lambda\pa\ol\pa|h^{-1}\pa_{K}h|^2_H
		&=|\nabla^H(h^{-1}\pa_{K}h)|^2_H+g^{\ol{j}i}<\nabla^H_i\nabla^H_{\ol{j}}(h^{-1}\pa_{K}h),h^{-1}\pa_{K}h>_H
		\\&+g^{\ol{j}i}<h^{-1}\pa_{K}h,\nabla^H_{\ol{i}}\nabla^H_j(h^{-1}\pa_{K}h))>_H.
\end{split}\end{equation}
Let $T_{\nabla_g}$ denote the torsion of $\nabla_g$, it follows
\begin{equation}\begin{split}
		\nabla^H_i\nabla^H_{\ol{j}}(h^{-1}\pa_{K,m}h)
		&=\nabla^H_i\ol\pa_{\ol{j}}(h^{-1}\pa_{K}h)_m
		\\&=\nabla^H_iF_{H,m\ol{j}}-\nabla^H_iF_{K,m\ol{j}}
		\\&=\nabla^H_mF_{H,i\ol{j}}
		+T_{im}^lF_{K,l\ol{j}}
		\\&+T_{im}^l\ol\pa_{\ol{j}}(h^{-1}\pa_{K}h)_l-\nabla^H_iF_{K,m\ol{j}},
\end{split}\end{equation}
where $T_{im}^n$ is given by $T_{\nabla_g}(\frac{\pa}{\pa z^i},\frac{\pa}{\pa z^m})=T_{im}^l\frac{\pa}{\pa z^l}$ and we have used the relation
\begin{equation}\begin{split}
		0&=\nabla^H_iF_{H,m\ol{j}}-\nabla^H_mF_{H,i\ol{j}}
		+\nabla^H_{\ol{j}}F_{H,im}
		-T_{im}^lF_{H,l\ol{j}}
		\\&=\nabla^H_iF_{H,m\ol{j}}-\nabla^H_mF_{H,i\ol{j}}
		-T_{im}^lF_{K,l\ol{j}}+T_{im}^l\ol\pa_{\ol{j}}(h^{-1}\pa_{K}h)_l,
\end{split}\end{equation}
which comes from the Bianchi identity, $T^{1,1}_{\nabla_g}=0$ and $T_{\nabla_g}(\frac{\pa}{\pa z^i},\frac{\pa}{\pa z^m})\in T^{1,0}N$. Hence
\begin{equation}\begin{split}\label{BB}
		&g^{\ol{j}i}<\nabla^H_i\nabla^H_{\ol{j}}(h^{-1}\pa_{K}h),h^{-1}\pa_{K}h>_H
		\\&\geq-C_1(|\nabla F_{K}|_H+|T_{\nabla_g}||F_{K}|_H+|\nabla\Lambda F_{H}|_H)|h^{-1}\pa_{K}h|_H
		\\&-C_1(|F_{K}|_H+|\Lambda F_{H}|_H)|h^{-1}\pa_{K}h|^2_H
		\\&-C_1|T||h^{-1}\pa_{K}h|_H|\nabla^H(h^{-1}\pa_{K}h)|_H,
\end{split}\end{equation}
where $C_1$ is positive constant independent of $K$ and $H$. Moreover, it holds that
\begin{equation}\begin{split}
		\nabla^H_{\ol{i}}\nabla^H_j(h^{-1}\pa_{K}h)_{m\alpha}^\beta
		&=\nabla^H_j\nabla^H_{\ol{i}}(h^{-1}\pa_{K}h)_{m\alpha}^\beta
		-R_{j\ol{i}m}^l(h^{-1}\pa_{K}h)_{l\alpha}^\beta
		\\&+F_{H,j\ol{i}\gamma}^\beta(h^{-1}\pa_{K}h)_{m\alpha}^\gamma-F_{H,j\ol{i}\alpha}^\gamma(h^{-1}\pa_{K}h)_{m\gamma}^\beta,
\end{split}\end{equation}
where $R_{j\ol{i}m}^l$ is given by $R_{\nabla_g}(\frac{\pa}{\pa z^j},\frac{\pa}{\pa\ol{z}^i})\frac{\pa}{\pa z^m}=R_{j\ol{i}m}^l\frac{\pa}{\pa z^l}$. Hence there is positive constant $C_2$ doesn't depend on $K$ and $H$ such that
\begin{equation}\begin{split}\label{CC}
		&g^{\ol{j}i}<h^{-1}\pa_{K}h,\nabla^H_{\ol{i}}\nabla^H_j(h^{-1}\pa_{K}h)>_H
		\\&\geq-C_2(|\nabla F_{K}|_H+|T_{\nabla_g}||F_{K}|_H+|\nabla\Lambda F_{H}|_H)|h^{-1}\pa_{K}h|_H
		\\&-C_2(|R_{\nabla_g}|+|F_{K}|_H+|\Lambda F_{H}|_H)|h^{-1}\pa_{K}h|^2_H
		\\&-C_2|T||h^{-1}\pa_{K}h|_H|\nabla^H(h^{-1}\pa_{K}h)|_H.
\end{split}\end{equation}
Combing $(\ref{AA})$, $(\ref{BB})$, $(\ref{CC})$ together, we can conclude
\begin{lemma}\label{DD}
	There is a positive constant $C_3$ independent of $K$ and $H$ such that
	\begin{equation}\begin{split}
			&\sqrt{-1}\Lambda\pa\ol\pa|h^{-1}\pa_{K}h|^2_H
			\\&\geq-C_3(|\nabla F_{K}|_H+|T_{\nabla_g}||F_{K}|_H+|\nabla\Lambda F_{H}|_H)|h^{-1}\pa_{K}h|_H
			\\&-C_3(|T_{\nabla_g}|^2+|R_{\nabla_X}|+|F_{K}|_H+|\Lambda F_{H}|_H)|h^{-1}\pa_{K}h|^2_H
			\\&+\frac{1}{2}|\nabla^H(h^{-1}\pa_{K}h)|^2_H.
	\end{split}\end{equation}
\end{lemma}
\section{Solvability of Hermitian-Einstein equations}
\begin{lemma}\label{tracelemma}
	For a compact Hermitian manifold $(N,J,g)$ with boundary and suppose $H$ is a solution of the following perturbed heat flow
	\begin{equation}\begin{split}\label{flow}
			\left\{ \begin{array}{ll}
				H^{-1}\frac{\pa H}{\pa t}=-2\left(\sqrt{-1}(\Lambda F_H-\frac{\tr\Lambda F_K}{\rank(E)}\id_E)+\epsilon\log(K^{-1}H)\right),\\
				H(0)=K,\ H|_{\pa N}=K|_{\pa N}.
			\end{array}\right.
	\end{split}\end{equation}
	Then we have $\det K=\det H$.
\end{lemma}
\begin{proof}
	We compute
	\begin{equation}\begin{split}\label{trt}
			&(\frac{\pa}{\pa t}-2\sqrt{-1}\Lambda\pa\ol\pa)\left(e^{2\epsilon t}
			\tr(\sqrt{-1}(\Lambda F_H-\frac{\tr\Lambda F_K}{\rank(E)}\id_E)+\epsilon\log(K^{-1}H))\right)
			\\&=2\epsilon e^{2\epsilon t}\tr\sqrt{-1}(\Lambda F_H-\frac{\tr\Lambda F_K}{\rank(E)}\id_E)
			+2\epsilon^2e^{2\epsilon t}\log\det(K^{-1}H)
			\\&+e^{2\epsilon t}(\frac{\pa}{\pa t}-2\sqrt{-1}\Lambda\pa\ol\pa)\tr\sqrt{-1}(\Lambda F_H-\frac{\tr\Lambda F_K}{\rank(E)}\id_E)
			\\&+\epsilon e^{2\epsilon t}(\frac{\pa}{\pa t}-2\sqrt{-1}\Lambda\pa\ol\pa)\log\det(K^{-1}H)
			\\&=e^{2\epsilon t}\frac{\pa}{\pa t}\tr\sqrt{-1}\Lambda F_{H}
			-2e^{2\epsilon t}\sqrt{-1}\Lambda\pa\ol\pa\tr\sqrt{-1}(\Lambda F_H-\frac{\tr\Lambda F_K}{\rank(E)}\id_E)
			\\&-2\epsilon e^{2\epsilon t}\sqrt{-1}\Lambda\pa\ol\pa\log\det(K^{-1}H)
			\\&=-2e^{2\epsilon t}\tr\left(\sqrt{-1}\Lambda\ol\pa_E\pa(\sqrt{-1}(\Lambda F_H-\frac{\tr\Lambda F_K}{\rank(E)}\id_E)+\epsilon\log(K^{-1}H))\right)
			\\&-2e^{2\epsilon t}\sqrt{-1}\Lambda\pa\ol\pa\tr\sqrt{-1}(\Lambda F_H-\frac{\tr\Lambda F_K}{\rank(E)}\id_E)
			-2\epsilon e^{2\epsilon t}\sqrt{-1}\Lambda\pa\ol\pa\log\det(K^{-1}H)
			\\&=0,
	\end{split}\end{equation}
	where we have used
	\begin{equation}\begin{split}
			\frac{\pa}{\pa t}\log\det(K^{-1}H)
			&=\tr(H^{-1}\frac{\pa H}{\pa t})
			\\&=-2\tr\sqrt{-1}(\Lambda F_H-\frac{\tr\Lambda F_K}{\rank(E)}\id_E)
			-2\epsilon\tr\log(K^{-1}H)
			\\&=-2\tr\sqrt{-1}(\Lambda F_H-\frac{\tr\Lambda F_K}{\rank(E)}\id_E)
			-2\epsilon\log\det(K^{-1}H),
	\end{split}\end{equation}
	and
	\begin{equation}\begin{split}
			\frac{\pa\tr F_{H}}{\pa t}
			&=\frac{\pa\tr\ol\pa_E\left((K^{-1}H)^{-1}\pa_K(K^{-1}H)\right)}{\pa t}
			\\&=\tr\ol\pa_E\pa_{H}\left((K^{-1}H)^{-1}\frac{\pa K^{-1}H}{\pa t}\right)
			\\&=-2\tr\ol\pa_E\pa_H\left(\sqrt{-1}(\Lambda F_H-\frac{\tr\Lambda F_K}{\rank(E)}\id_E)+\epsilon\log(K^{-1}H)\right).
	\end{split}\end{equation}
	\par In addition, we note on $N\times\{t=0\}$ and $\ol\pa N\times[0,t)$,
	\begin{equation}\begin{split}
			e^{2\epsilon t}
			\tr\left(\sqrt{-1}(\Lambda F_H-\frac{\tr\Lambda F_K}{\rank(E)}\id_E)+\epsilon\log(K^{-1}H)\right)=0,
	\end{split}\end{equation}
	it follows from the maximum principle that on $N\times[0,t)$,
	\begin{equation}\begin{split}
			e^{2\epsilon t}
			\tr\left(\sqrt{-1}(\Lambda F_H-\frac{\tr\Lambda F_K}{\rank(E)}\id_E)+\epsilon\log(K^{-1}H)\right)=0.
	\end{split}\end{equation}
	Thus we finish the proof by noting
	\begin{equation}\begin{split}
			\frac{\pa}{\pa t}\log\det(K^{-1}H)
			&=-2\tr\left(\sqrt{-1}(\Lambda F_H-\frac{\tr\Lambda F_K}{\rank(E)}\id_E)+\epsilon\log(K^{-1}H)\right)
			\\&=0.
	\end{split}\end{equation}
\end{proof}
\begin{remark}\label{flowremark}
	It seems more natural to consider
	\begin{equation}\begin{split}
			\left\{ \begin{array}{ll}
				H^{-1}\frac{\pa H}{\pa t}=-2(\sqrt{-1}\Lambda F_{H}^\perp+\epsilon\log(K^{-1}H)),\\
				H(0)=K.
			\end{array}\right.
	\end{split}\end{equation}
	For this flow, one may not run the above argument to obtain Lemma \ref{tracelemma} since
	\begin{equation}\begin{split}
			&(\frac{\pa}{\pa t}-2\sqrt{-1}\Lambda\pa\ol\pa)\left(e^{2\epsilon t}
			\tr(\sqrt{-1}\Lambda F_{H}^\perp+\epsilon\log(K^{-1}H))\right)
			\\&=(\frac{\pa}{\pa t}-2\sqrt{-1}\Lambda\pa\ol\pa)\epsilon e^{2\epsilon t}\log\det(K^{-1}H)
			\\&=\epsilon e^{2\epsilon t}\left(2\epsilon\log\det(K^{-1}H)+\tr(H^{-1}\frac{\pa H}{\pa t})-2\sqrt{-1}\Lambda\pa\ol\pa\log\det(K^{-1}H)\right)
			\\&=-2\epsilon e^{2\epsilon t}\sqrt{-1}\Lambda\pa\ol\pa\log\det(K^{-1}H).
	\end{split}\end{equation}
\end{remark}
\begin{lemma}\label{estimatelemma}
	Let $H$, $\tilde{H}$ be two solutions to the perturbed heat flow $(\ref{flow})$ on a compact Hermitian manifold $(N,J,g)$, then it holds
	\begin{equation}\begin{split}\label{distancet}
			(\frac{\pa}{\pa t}-2\sqrt{-1}\Lambda\pa\ol\pa)\sigma(H,\tilde{H})
			\leq-|\ol\pa_E\tilde{h}\circ\tilde{h}^{-\frac{1}{2}}|^2_H-|\ol\pa_E\tilde{h}^{-1}\circ\tilde{h}^{\frac{1}{2}}|^2_{\tilde{H}},
	\end{split}\end{equation}
	\begin{equation}\begin{split}\label{loghlaplace}
			(\frac{\pa}{\pa t}-2\sqrt{-1}\Lambda\pa\ol\pa)|\log h|^2&\leq4|\Lambda F_{K}^\perp||\log h|-4\epsilon|\log h|^2,
	\end{split}\end{equation}
	\begin{equation}\begin{split}\label{modulet}
			(\frac{\pa}{\pa t}-2\sqrt{-1}\Lambda\pa\ol\pa)|\sqrt{-1}\Lambda F_{H}^\perp+\epsilon\log h|_H^2
			&\leq-4|\ol\pa_E(\sqrt{-1}\Lambda F_{H}^\perp+\epsilon\log h)|_H^2,
	\end{split}\end{equation}
	where $h=K^{-1}H$, $\tilde{h}=H^{-1}\tilde{H}$ and $\sigma(H,\tilde{H})=\tr(\tilde{h}+\tilde{h}^{-1})-2\rank(E)$.
\end{lemma}
\begin{proof}
	Based on Lemma \ref{tracelemma}, the proof is a direct calculation, see \cite{ZZZ2021} for example.
\end{proof}
Using Lemmas \ref{tracelemma}, \ref{estimatelemma} and following \cite{Do1992,Zh2005}, we can conclude the existence result on Dirichlet problems of perturbed $\HE$ equations, see \cite[Theorem 5.1]{ZZZ2021}.
\begin{proposition}\label{Dirichletsolution}
	Let $(E,\ol\pa_E,K)$ be a Hermitian holomorphic vector bundle on a compact Hermitian manifold $(N,J,g)$ with nonempty boundary. For any $\epsilon\geq0$, there exists a unique Hermitian metric $H_\epsilon$ such that
	\begin{equation}\begin{split}
			\sqrt{-1}\Lambda F_{H_\epsilon}^\perp+\epsilon\log(K^{-1}H_\epsilon)=0,\ H_\epsilon|_{\pa N}=K|_{\pa N}.
	\end{split}\end{equation}
	Moreover if $\epsilon>0$, it holds
	\begin{equation}\begin{split}
			|\log(K^{-1}H_\epsilon)|\leq\epsilon^{-1}\sup\limits_{N}|\Lambda F_{K}^\perp|.
	\end{split}\end{equation}
\end{proposition}
\begin{proof}[\textup{\textbf{Proof of Theorem \ref{thm2}}}]
We may take a family of compact submanifolds $\{M_j\}_{j\in\mathbb{N}}$ with boundaries such that
\begin{equation}\begin{split}
		M_j\xrightarrow{j\rightarrow\infty}M\setminus\Sigma,\ M_j\subset M_{j^{'}},\ \forall j<j^{'}.
\end{split}\end{equation}
By Proposition \ref{Dirichletsolution}, let $H_{\epsilon,j}$ be the unique Hermitian metric on $E|_{M_j}$ such that
\begin{equation}\begin{split}
		\sqrt{-1}\Lambda F_{H_{\epsilon,j}}^\perp+\epsilon\log(K^{-1}H_{\epsilon,j})=0,\
		H_{\epsilon,j}|_{\pa M_j}=K|_{\pa M_j}.
\end{split}\end{equation}
By setting $h_{\epsilon,j}=K^{-1}H_{\epsilon,j}$ and making use of $(\ref{keyequality})$, we obtain
\begin{equation}\begin{split}
		\int_{M_j}\sqrt{-1}\Lambda\pa\ol\pa|\log h_{\epsilon,j}|^2\dvol_g
		&=2\int_{M_j}<\sqrt{-1}\Lambda(F_{K}-F_{H_{\epsilon,j}}),\log h_{\epsilon,j}>\dvol_g
		\\&+2\int_{M_j}<\Theta[\log h_{\epsilon,j}](\ol\pa_E\log h_{\epsilon,j}),\ol\pa_E\log h_{\epsilon,j}>\dvol_g.
\end{split}\end{equation}
Note that
\begin{equation}\begin{split}
		<\sqrt{-1}\Lambda(F_{K}-F_{H_{\epsilon,j}}),\log h_{\epsilon,j}>
		&=<\sqrt{-1}\Lambda(F_{K}^\perp-F_{H_{\epsilon,j}}^\perp),\log h_{\epsilon,j}>
		\\&=<\sqrt{-1}\Lambda F_{K}^\perp,\log h_{\epsilon,j}>
		+\epsilon|\log h_{\epsilon,j}|^2,
\end{split}\end{equation}
the boundary condition and $\sqrt{-1}\pa\ol\pa\omega_g^{\dim_{\mathbb{C}}M-1}\geq0$ imply
\begin{equation}\begin{split}
		\int_{M_j}\sqrt{-1}\Lambda\pa\ol\pa|\log h_{\epsilon,j}|^2\dvol_g
		&=\int_{M_j}\sqrt{-1}\pa\ol\pa|\log h_{\epsilon,j}|^2\wedge\frac{\omega_g^{\dim_{\mathbb{C}}M-1}}{(\dim_{\mathbb{C}}M-1)!}
		\\&=\int_{M_j}\sqrt{-1}\pa(\ol\pa|\log h_{\epsilon,j}|^2\wedge\frac{\omega_g^{\dim_{\mathbb{C}}M-1}}{(\dim_{\mathbb{C}}M-1)!})
		\\&+\int_{M_j}\sqrt{-1}\ol\pa(|\log h_{\epsilon,j}|^2\wedge\frac{\pa\omega_g^{\dim_{\mathbb{C}}M-1}}{(\dim_{\mathbb{C}}M-1)!})
		\\&-\int_{M_j}\sqrt{-1}|\log h_{\epsilon,j}|^2\wedge\frac{\ol\pa\pa\omega_g^{\dim_{\mathbb{C}}M-1}}{(\dim_{\mathbb{C}}M-1)!}
		\\&\leq0.
\end{split}\end{equation}
It follows from above that
\begin{equation}\begin{split}\label{identity}
		&\int_{M_j}<\Theta[\log h_{\epsilon,j}](\ol\pa_E\log h_{\epsilon,j}),\ol\pa_E\log h_{\epsilon,j}>\dvol_g
		\\&\leq-\epsilon\int_{M_j}|\log h_{\epsilon,j}|^2\dvol_g
		-\int_{M_j}<\sqrt{-1}\Lambda F_{K}^\perp,\log h_{\epsilon,j}>\dvol_g.
\end{split}\end{equation}
\par Below we always use $C_1, C_2, C_3,...$ to denote constants independent of $j$. For $\epsilon>0$, Proposition \ref{Dirichletsolution} implies
\begin{equation}\begin{split}\label{c1a}
		|\log h_{\epsilon,j}|\leq\epsilon^{-1}\sup\limits_{M\setminus\Sigma}|\Lambda F_{K}^\perp|.
\end{split}\end{equation}
Moreover, a direct computation gives us that
\begin{equation}\begin{split}\label{c1b}
		\sqrt{-1}\Lambda\pa\ol\pa\tr h_{\epsilon,j}
		&=<\sqrt{-1}\Lambda(F_K-F_{H_{\epsilon,j}}),h_{\epsilon,j}>+|h_{\epsilon,j}^{-\frac{1}{2}}\pa_Kh_{\epsilon,j}|^2,
\end{split}\end{equation}
and also Lemma \ref{DD} implies
\begin{equation}\begin{split}\label{c1c}
		&\sqrt{-1}\Lambda\pa\ol\pa|h^{-1}_{\epsilon,j}\pa_{K}h_{\epsilon,j}|^2_{H_{\epsilon,j}}
		\\&\geq-C_1(|\nabla F_{K}|_{H_{\epsilon,j}}+|T_{\nabla_g}||F_{K}|_{H_{\epsilon,j}}+|\nabla\Lambda F_{H_{\epsilon,j}}|_{H_{\epsilon,j}})|h^{-1}_{\epsilon,j}\pa_{K}h_{\epsilon,j}|_{H_{\epsilon,j}}
		\\&-C_1(|T_{\nabla_g}|^2+|R_{\nabla_g}|+|F_{K}|_{H_{\epsilon,j}}+|\Lambda F_{H_{\epsilon,j}}|_{H_{\epsilon,j}})|h^{-1}_{\epsilon,j}\pa_{K}h_{\epsilon,j}|^2_{H_{\epsilon,j}}
		\\&+\frac{1}{2}|\nabla^{H_{\epsilon,j}}(h^{-1}_{\epsilon,j}\pa_{K}h_{\epsilon,j})|^2_{H_{\epsilon,j}},
\end{split}\end{equation}
Due to $(\ref{c1a})$ and
\begin{equation}\begin{split}
		\sqrt{-1}\Lambda F_{H_{\epsilon,j}}=\frac{\sqrt{-1}\Lambda\tr F_{K}}{\rank(E)}+\epsilon\log h_{\epsilon,j},
\end{split}\end{equation}
we can write $(\ref{c1b})$ and $(\ref{c1c})$ as
\begin{equation}\begin{split}\label{c1d}
		\sqrt{-1}\Lambda\pa\ol\pa\tr h_{\epsilon,j}
		&\geq C_2|h_{\epsilon,j}^{-1}\pa_Kh_{\epsilon,j}|^2-C_3,
\end{split}\end{equation}
\begin{equation}\begin{split}\label{c1e}
		\sqrt{-1}\Lambda\pa\ol\pa|h^{-1}_{\epsilon,j}\pa_{K}h_{\epsilon,j}|^2_{H_{\epsilon,j}}
		&\geq\frac{1}{2}|\nabla^{H_{\epsilon,j}}(h^{-1}_{\epsilon,j}\pa_{K}h_{\epsilon,j})|^2_{H_{\epsilon,j}}-C_4|h^{-1}_{\epsilon,j}\pa_{K}h_{\epsilon,j}|^2-C_4.
\end{split}\end{equation}
For any closed subset $\Omega\subset M_{j_\ast}\subset M\setminus\Sigma$ with large $j_\ast$, we set 
\begin{equation}\begin{split}
		\Omega_1=\{x\in M_{j_\ast}, d(x,\Omega)\leq4^{-1}\dist(\Omega,\pa M_{j_\ast})\},
\end{split}\end{equation}
\begin{equation}\begin{split}
		\Omega_2=\{x\in M_{j_\ast}, \dist(x,\Omega)\leq2^{-1}d(\Omega,\pa M_{j_\ast})\}.
\end{split}\end{equation}
Let $0\leq\eta_1\leq\eta_2\leq1$ be two nonnegative cut-off functions with $\eta_1\equiv1$ on $\Omega$, $\eta_1\equiv0$ on $M_{j_\ast}\setminus\Omega_1$, $\eta_2\equiv1$ on $\Omega_1$, $\eta_2\equiv0$ on $M_{j_\ast}\setminus\Omega_2$ and $|d\eta_i|^2+|\Lambda\pa\ol\pa\eta_i|\leq\tilde{C}$ for $i=1,2$, where $\tilde{C}$ depends only on $\dist^{-2}(\Omega,\pa M_{j_\ast})$. We define the quantity
\begin{equation}\begin{split}
		f=\eta_1^2|h^{-1}_{\epsilon,j}\pa_{K}h_{\epsilon,j}|^2_{H_{\epsilon,j}}+\hat{C}\eta_2^2\tr h_{\epsilon,j},
\end{split}\end{equation}
for constant $\hat{C}$ to be determined later. By $(\ref{c1d})$ and $(\ref{c1e})$, we get
\begin{equation}\begin{split}
		\sqrt{-1}\Lambda\pa\ol\pa f
		&\geq\eta_1^2(\frac{1}{2}|\nabla^{H_{\epsilon,j}}(h^{-1}_{\epsilon,j}\pa_{K}h_{\epsilon,j})|^2_{H_{\epsilon,j}}-C_4|h^{-1}_{\epsilon,j}\pa_{K}h_{\epsilon,j}|^2-C_4)
		\\&+\hat{C}\eta_2^2(C_2|h_{\epsilon,j}^{-1}\pa_Kh_{\epsilon,j}|^2-C_3)
		\\&+|h^{-1}_{\epsilon,j}\pa_{K}h_{\epsilon,j}|^2_{H_{\epsilon,j}}\sqrt{-1}\Lambda\pa\ol\pa\eta_1^2+\hat{C}\tr h_{\epsilon,j}\sqrt{-1}\Lambda\pa\ol\pa\eta_2^2
		\\&+2<\pa\eta_1^2,\pa|h^{-1}_{\epsilon,j}\pa_{K}h_{\epsilon,j}|^2_{H_{\epsilon,j}}>+2\hat{C}<\pa\eta_2^2,\pa\tr h_{\epsilon,j}>.
\end{split}\end{equation}
By the inequality $ab\leq\chi a^2+\frac{1}{4\chi^2}b^2$ for $a,b,\chi>0$, we have
\begin{equation}\begin{split}
		&2<\pa\eta_1^2,\pa|h^{-1}_{\epsilon,j}\pa_{K}h_{\epsilon,j}|^2_{H_{\epsilon,j}}>+2\hat{C}<\pa\eta_2^2,\pa\tr h_{\epsilon,j}>
		\\&\geq-8\eta_1|\pa\eta_1||h^{-1}_{\epsilon,j}\pa_{K}h_{\epsilon,j}|_{H_{\epsilon,j}}|\nabla^{H_{\epsilon,j}}(h^{-1}_{\epsilon,j}\pa_{K}h_{\epsilon,j})|_{H_{\epsilon,j}}-4\hat{C}\eta_2|\pa\eta_2||\tr\nabla^{H_{\epsilon,j}}h_{\epsilon,j}|
		\\&\geq-\frac{1}{2}\eta_1^2|\nabla^{H_{\epsilon,j}}(h^{-1}_{\epsilon,j}\pa_{K}h_{\epsilon,j})|_{H_{\epsilon,j}}^2-32|\pa\eta_1|^2|h^{-1}_{\epsilon,j}\pa_{K}h_{\epsilon,j}|_{H_{\epsilon,j}}^2
		\\&-2\eta_2^2|\tr\nabla^{H_{\epsilon,j}}h_{\epsilon,j}|^2-2\hat{C}^2|\pa\eta_2|^2,
\end{split}\end{equation}
and it holds
\begin{equation}\begin{split}
		\sqrt{-1}\Lambda\pa\ol\pa f
		&\geq-\eta_1^2(C_4|h^{-1}_{\epsilon,j}\pa_{K}h_{\epsilon,j}|^2+C_4)
		+\hat{C}\eta_2^2(C_2|h_{\epsilon,j}^{-1}\pa_Kh_{\epsilon,j}|^2-C_3)
		\\&+2|h^{-1}_{\epsilon,j}\pa_{K}h_{\epsilon,j}|^2_{H_{\epsilon,j}}(\eta_1\sqrt{-1}\Lambda\pa\ol\pa\eta_1+|\pa\eta_1|^2)
		\\&+2\hat{C}\tr h_{\epsilon,j}(\eta_2\sqrt{-1}\Lambda\pa\ol\pa\eta_2+|\pa\eta_2|^2)
		\\&-32|\pa\eta_1|^2|h^{-1}_{\epsilon,j}\pa_{K}h_{\epsilon,j}|_{H_{\epsilon,j}}^2
		-2\eta_2^2|\tr\nabla^{H_{\epsilon,j}}h_{\epsilon,j}|^2-2\hat{C}^2|\pa\eta_2|^2
		\\&\geq-\eta_2^2(C_4|h^{-1}_{\epsilon,j}\pa_{K}h_{\epsilon,j}|^2+C_4)
		+\hat{C}\eta_2^2(C_2|h_{\epsilon,j}^{-1}\pa_Kh_{\epsilon,j}|^2-C_3)
		\\&-2\tilde{C}C_5\eta_2^2|h^{-1}_{\epsilon,j}\pa_{K}h_{\epsilon,j}|^2-2\hat{C}\tilde{C}C_6\eta_2^2
		\\&-32\tilde{C}C_5|h^{-1}_{\epsilon,j}\pa_{K}h_{\epsilon,j}|^2
		-2C_7\eta_2^2|h^{-1}_{\epsilon,j}\pa_{K}h_{\epsilon,j}|^2-2\hat{C}^2\tilde{C}
		\\&=\eta_2^2|h^{-1}_{\epsilon,j}\pa_{K}h_{\epsilon,j}|^2(\hat{C}C_2-C_4-2\tilde{C}C_5-32\tilde{C}C_5-2C_7)
		\\&-C_4-\hat{C}C_3-2\hat{C}\tilde{C}C_6-2\hat{C}^2\tilde{C}.
\end{split}\end{equation}
By taking $\hat{C}$ such that $\hat{C}C_2-C_4-2\tilde{C}C_5-32\tilde{C}C_5-2C_7=1$, we have
\begin{equation}\begin{split}
		\sqrt{-1}\Lambda\pa\ol\pa f\geq\eta_2^2|h^{-1}_{\epsilon,j}\pa_{K}h_{\epsilon,j}|^2-C_8.
\end{split}\end{equation}
We assume $x_0\in\Omega_2$ such that $f(x_0)=\max\limits_{M_{j_\ast}}f$. Obviously, it holds
\begin{equation}\begin{split}
		\max\limits_{\Omega}|h^{-1}_{\epsilon,j}\pa_{K}h_{\epsilon,j}|^2_{H_{\epsilon,j}}
		&=\max\limits_{\Omega}\eta_1^2|h^{-1}_{\epsilon,j}\pa_{K}h_{\epsilon,j}|^2_{H_{\epsilon,j}}
		\\&\leq\max\limits_{\Omega}f
		\\&\leq f(x_0).
\end{split}\end{equation}
If $x_0\in\Omega_2\setminus\Omega_1$, we have
\begin{equation}\begin{split}
\max\limits_{\Omega}|h^{-1}_{\epsilon,j}\pa_{K}h_{\epsilon,j}|^2_{H_{\epsilon,j}}
&\leq f(x_0)
\\&=\hat{C}\eta_2^2\tr h_{\epsilon,j}(x_0)
\\&\leq C_9.
	\end{split}\end{equation}
If $x_0\in\Omega_1$, we have
\begin{equation}\begin{split}
		\max\limits_{\Omega}|h^{-1}_{\epsilon,j}\pa_{K}h_{\epsilon,j}|^2_{H_{\epsilon,j}}
		&\leq f(x_0)
		\\&=\eta_1^2|h^{-1}_{\epsilon,j}\pa_{K}h_{\epsilon,j}|^2_{H_{\epsilon,j}}(x_0)+\hat{C}\eta_2^2\tr h_{\epsilon,j}(x_0)
		\\&\leq C_{10}\left(\eta_2^2|h^{-1}_{\epsilon,j}\pa_{K}h_{\epsilon,j}|^2(x_0)+1\right)
		\\&\leq C_{10}(C_8+1).
\end{split}\end{equation}
In both cases, we conclude the uniform estimate of $\{h_{\epsilon,j}\}_{j\in\mathbb{N}}$ up to arbitrary orders on $\Omega$ by standard elliptic regularity. We may take $j_\epsilon\rightarrow\infty$ so that $h_{\epsilon,j_\epsilon}\xrightarrow{j_\epsilon\rightarrow\infty}h_\epsilon$ in $C^\infty_{\loc}$-topology and
\begin{equation}\begin{split}\label{epsilonequation}
		\sqrt{-1}\Lambda F_{H_\epsilon}^\perp+\epsilon\log h_{\epsilon}=0,\  \epsilon>0.
\end{split}\end{equation}
\par Now we compute
\begin{equation}\begin{split}
		\sqrt{-1}\Lambda\pa\ol\pa\log(\frac{\tr h_{\epsilon,j_\epsilon}}{\rank(E)})
		&=\frac{<\sqrt{-1}\Lambda(F_K-F_{H_{\epsilon,j_\epsilon}}),h_{\epsilon,j_\epsilon}>}{\tr h_{\epsilon,j_\epsilon}}
		\\&-\frac{|\tr\pa_Kh_{\epsilon,j_\epsilon}|^2}{(\tr h_{\epsilon,j_\epsilon})^2}+\frac{|h_{\epsilon,j_\epsilon}^{-\frac{1}{2}}\circ\pa_Kh_{\epsilon,j_\epsilon}|^2}{\tr h_{\epsilon,j_\epsilon}}
		\\&\geq\frac{<\sqrt{-1}\Lambda(F_K^\perp-F_{H_{\epsilon,j_\epsilon}}^\perp),h_{\epsilon,j_\epsilon}>}{\tr h_{\epsilon,j_\epsilon}}
		\\&=\frac{<\sqrt{-1}\Lambda F_K^\perp+\epsilon\log h_{\epsilon,j_\epsilon},h_{\epsilon,j_\epsilon}>}{\tr h_{\epsilon,j_\epsilon}}
		\\&\geq\frac{<\sqrt{-1}\Lambda F_K^\perp,h_{\epsilon,j_\epsilon}>}{\tr h_{\epsilon,j_\epsilon}}
		\\&\geq-|\Lambda F_K^\perp|,
\end{split}\end{equation}
where we have used $\frac{|\tr\pa_Kh_{\epsilon,j_\epsilon}|^2}{\tr h_{\epsilon,j_\epsilon}}\leq|h_{\epsilon,j_\epsilon}^{-\frac{1}{2}}\circ\pa_Kh_{\epsilon,j_\epsilon}|^2$, and $\tr(\log h_{\epsilon,j}\circ h_{\epsilon,j})\geq\tr\log h_{\epsilon,j}=0$ due to $\det K=\det H_{\epsilon,j}$. For any $p>0$, it holds
\begin{equation}\begin{split}
		&-\sup\limits_{M\setminus\Sigma}|\Lambda F_K^\perp|\int_{M_j}(\log(\frac{\tr h_{\epsilon,j_\epsilon}}{\rank(E)}))^p\dvol_g
		\\&\leq\int_{M_j}
		\sqrt{-1}\Lambda\pa\ol\pa\log(\frac{\tr h_{\epsilon,j_\epsilon}}{\rank(E)})(\log(\frac{\tr h_{\epsilon,j_\epsilon}}{\rank(E)}))^p\dvol_g
		\\&=\int_{M_j}
		\sqrt{-1}\pa\ol\pa\log(\frac{\tr h_{\epsilon,j_\epsilon}}{\rank(E)})(\log(\frac{\tr h_{\epsilon,j_\epsilon}}{\rank(E)}))^p\wedge\frac{\omega_g^{\dim_{\mathbb{C}}M-1}}{(\dim_{\mathbb{C}}M-1)!}
		\\&=\int_{M_j}
		\sqrt{-1}\ol\pa\log(\frac{\tr h_{\epsilon,j_\epsilon}}{\rank(E)})\wedge\pa(\log(\frac{\tr h_{\epsilon,j_\epsilon}}{\rank(E)}))^p\wedge\frac{\omega_g^{\dim_{\mathbb{C}}M-1}}{(\dim_{\mathbb{C}}M-1)!}
		\\&+\int_{M_j}
		\sqrt{-1}\ol\pa\log(\frac{\tr h_{\epsilon,j_\epsilon}}{\rank(E)})\wedge(\log(\frac{\tr h_{\epsilon,j_\epsilon}}{\rank(E)}))^p\wedge\frac{\pa\omega_g^{\dim_{\mathbb{C}}M-1}}{(\dim_{\mathbb{C}}M-1)!}
		\\&=-p\int_{M_j}(\log(\frac{\tr h_{\epsilon,j_\epsilon}}{\rank(E)}))^{p-1}|\pa\log(\frac{\tr h_{\epsilon,j_\epsilon}}{\rank(E)})|^2\dvol_g,
\end{split}\end{equation}
where we haved used $\log(\frac{\tr h_{\epsilon,j_\epsilon}}{\rank(E)})\geq\log(\frac{\tr h_{\epsilon,j_\epsilon}}{\rank(E)})|_{\pa M_j}=0$ and $\sqrt{-1}\pa\ol\pa\omega_g^{\dim_{\mathbb{C}}M-1}\geq0$. Then by extending $\log(\frac{\tr h_{\epsilon,j_\epsilon}}{\rank(E)})$ to a Lipschitz function on $M$ and using Moser iteration based on
\begin{equation}\begin{split}
	\int_{M_j}|\pa(\log(\frac{\tr h_{\epsilon,j_\epsilon}}{\rank(E)}))^{\frac{p+1}{2}}|^2\dvol_g
	&\leq\frac{(p+1)^2}{4p}\sup\limits_{M\setminus\Sigma}|\Lambda F_K^\perp|\int_{M_j}(\log(\frac{\tr h_{\epsilon,j_\epsilon}}{\rank(E)}))^p\dvol_g,
\end{split}\end{equation}
 we conclude
\begin{equation}\begin{split}
		\log(\frac{\tr h_{\epsilon,j_\epsilon}}{\rank(E)})
		\leq C_{11}(1+\int_{M_{j_\epsilon}}\log(\frac{\tr h_{\epsilon,j_\epsilon}}{\rank(E)})\dvol_{g}),
\end{split}\end{equation}
where we have used the fact that $g$ is mutually bounded with a Hermitian metric on $M$. It is also easy to see $\log(\frac{\tr h_{\epsilon,j_\epsilon}}{\rank(E)})\leq|\log h_{\epsilon,j_\epsilon}|\leq\rank(E)^{\frac{3}{2}}\log\tr h_{\epsilon,j_\epsilon}$ and hence it follows
\begin{equation}\begin{split}
		|\log h_{\epsilon,j_\epsilon}|
		&\leq\rank(E)^{\frac{3}{2}}\left(\log(\frac{\tr h_{\epsilon,j_\epsilon}}{\rank(E)})
		+\log\rank(E)\right)
		\\&\leq\rank(E)^{\frac{3}{2}}\left(C_{11}+C_{11}\int_{M_{j_\epsilon}}\log(\frac{\tr h_{\epsilon,j_\epsilon}}{\rank(E)})\dvol_{g}+\log\rank(E)
		\right)
		\\&\leq\rank(E)^{\frac{3}{2}}\left(C_{11}+C_{11}\int_{M_{j_\epsilon}}|\log h_{\epsilon,j_\epsilon}|\dvol_{g}+\log\rank(E)
		\right)
		\\&=C_{12}(1+\int_{M_{j_\epsilon}}|\log h_{\epsilon,j_\epsilon}|\dvol_{g}).
\end{split}\end{equation}
By Lebesgue theorem, we then arrive at 
\begin{equation}\begin{split}
		|\log h_\epsilon|
		&=\lim\limits_{j_\epsilon\rightarrow\infty}|\log h_{\epsilon,j_\epsilon}|
		\\&\leq\lim\limits_{j_\epsilon\rightarrow\infty}C_{12}(1+\int_{M_{j_\epsilon}}|\log h_{\epsilon,j_\epsilon}|\dvol_g)
		\\&=C_{12}(1+\int_{M\setminus\Sigma}|\log h_{\epsilon}|\dvol_g).
\end{split}\end{equation}
\par In the remaining of the proof, we may assume $\int_{M\setminus\Sigma}|\log h_{\epsilon}|\dvol_g>1$ and thus
\begin{equation}\begin{split}\label{mean}
		|\log h_\epsilon|
		\leq A\int_{M\setminus\Sigma}|\log h_{\epsilon}|\dvol_g,
\end{split}\end{equation}
for a universal constant $A$. To prove the existence of an $\AHE$ structure, it suffices to prove \begin{equation}\begin{split}
		\sup\limits_{M\setminus\Sigma}|\Lambda F_{H_{\epsilon}}^\perp|=\sup\limits_{M\setminus\Sigma}|\epsilon\log h_{\epsilon}|\xrightarrow{\epsilon\rightarrow0}0.
\end{split}\end{equation}
We may assume there exists a sequence $\epsilon\rightarrow0$ and a constant $B>0$ such that
\begin{equation}\begin{split}
		||\log h_\epsilon||_{L^1(M\setminus \Sigma)}\xrightarrow{\epsilon\rightarrow0}\infty,
\end{split}\end{equation}
\begin{equation}\begin{split}
		|\Lambda F_{H_{\epsilon}}^\perp|=|\epsilon\log h_{\epsilon}|\geq B.
\end{split}\end{equation}
Set $\widetilde{\log h_{\epsilon}}=||\log h_\epsilon||_{L^1(M\setminus \Sigma)}^{-1}\log h_{\epsilon}$ and it follows from $(\ref{mean})$ that
\begin{equation}\begin{split}
		|\widetilde{\log h_{\epsilon}}|\leq A,\
		||\widetilde{\log h_{\epsilon}}||_{L^1(M\setminus\Sigma)}=1.
\end{split}\end{equation}
Using $(\ref{identity})$, we have
\begin{equation}\begin{split}
		&\epsilon||\log h_{\epsilon,j_\epsilon}||_{L^1(M_{j_\epsilon})}^{-1}\int_{M_{j_\epsilon}}|\log h_{\epsilon,j_\epsilon}|^2\dvol_g
		\\&+\int_{M_{j_\epsilon}}||\log h_{\epsilon,j_\epsilon}||_{L^1(M_{j_\epsilon})}<\Theta[\log h_{\epsilon,j_\epsilon}](\ol\pa_E\widetilde{\log h_{\epsilon,j_\epsilon}}),\ol\pa_E\widetilde{\log h_{\epsilon,j_\epsilon}}>\dvol_g
		\\&=-\int_{M_{j_\epsilon}}<\sqrt{-1}\Lambda F_{K}^\perp,\widetilde{\log h_{\epsilon,j_\epsilon}}>\dvol_g,
\end{split}\end{equation}
where $\widetilde{\log h_{\epsilon,j_\epsilon}}=||\log h_{\epsilon,j_\epsilon}||_{L^1(M_{j_\epsilon})}^{-1}\log h_{\epsilon,j_\epsilon}$. Given any $\delta>0$, there exists large $J_\delta$ such that if $j_\epsilon^{'}\geq j_\epsilon\geq J_\delta$, it holds
\begin{equation}\begin{split}
		\vol(M_{j_\epsilon^{'}}\setminus M_{j_\epsilon})\leq\delta.
\end{split}\end{equation}
We may assume $|\widetilde{\log h_{\epsilon,j_\epsilon}}|\leq A$ and therefore we deduce
\begin{equation}\begin{split}
		&\epsilon||\log h_{\epsilon,j_\epsilon^{'}}||_{L^1(M_{j_\epsilon^{'}})}^{-1}\int_{M_{j_\epsilon}}|\log h_{\epsilon,j_\epsilon^{'}}|^2\dvol_g
		\\&+\int_{M_{j_\epsilon}}||\log h_{\epsilon,j_\epsilon^{'}}||_{L^1(M_{j_\epsilon^{'}})}<\Theta[\log h_{\epsilon,j_\epsilon^{'}}](\ol\pa_E\widetilde{\log h_{\epsilon,j_\epsilon^{'}}}),\ol\pa_E\widetilde{\log h_{\epsilon,j_\epsilon^{'}}}>\dvol_g
		\\&\leq-(\int_{M_{j_\epsilon^{'}}\setminus M_{j_\epsilon}}
		+\int_{M_{j_\epsilon}})<\sqrt{-1}\Lambda F_{K}^\perp,\widetilde{\log h_{\epsilon,j_\epsilon^{'}}}>\dvol_g
		\\&\leq\vol(M_{j_\epsilon^{'}}\setminus M_{j_\epsilon})
		\sup\limits_{M_{j_{\epsilon}}}|\widetilde{\log h_{\epsilon,j_\epsilon^{'}}}|
		\sup\limits_{M\setminus\Sigma}|\Lambda F_K^\perp|
		-\int_{M_{j_\epsilon}}<\sqrt{-1}\Lambda F_{K}^\perp,\widetilde{\log h_{\epsilon,j_\epsilon^{'}}}>\dvol_g
		\\&\leq\delta A\sup\limits_{M\setminus\Sigma}|\Lambda F_{K}^\perp|
		-\int_{M_{j_\epsilon}}<\sqrt{-1}\Lambda F_{K}^\perp,\widetilde{\log h_{\epsilon,j_\epsilon^{'}}}>\dvol_g.
\end{split}\end{equation}
Taking $j_\epsilon^{'}\rightarrow\infty$ firstly and $j_\epsilon\rightarrow\infty$ secondly, we can arrive at
\begin{equation}\begin{split}
		&\epsilon||\log h_{\epsilon}||_{L^1(M\setminus\Sigma)}^{-1}\int_{M\setminus\Sigma}|\log h_{\epsilon}|^2\dvol_g
		\\&+\int_{M\setminus\Sigma}||\log h_{\epsilon}||_{L^1(M\setminus\Sigma)}<\Theta[\log h_{\epsilon}](\ol\pa_E\widetilde{\log h_{\epsilon}}),\ol\pa_E\widetilde{\log h_{\epsilon}}>\dvol_g
		\\&\leq\delta A\sup\limits_{M\setminus\Sigma}|\Lambda F_{K}^\perp|
		-\int_{M\setminus\Sigma}<\sqrt{-1}\Lambda F_{K}^\perp,\widetilde{\log h_{\epsilon}}>\dvol_g.
\end{split}\end{equation}
Note that
\begin{equation}\begin{split}
		&\epsilon||\log h_{\epsilon}||_{L^1(M\setminus\Sigma)}^{-1}\int_{M\setminus\Sigma}|\log h_{\epsilon}|^2\dvol_g
		\\&\geq B|\log h_{\epsilon}|^{-1}||\log h_{\epsilon}||_{L^1(M\setminus\Sigma)}^{-1}\int_{M\setminus\Sigma}|\log h_{\epsilon}|^2\dvol_g
		\\&\geq BA^{-1}||\log h_{\epsilon}||_{L^1(M\setminus\Sigma)}^{-2}\int_{M\setminus\Sigma}|\log h_{\epsilon}|^2\dvol_g
		\\&\geq B(A\vol(M\setminus\Sigma))^{-1},
\end{split}\end{equation}
this and the fact $\delta$ is arbitrary furthermore give us
\begin{equation}\begin{split}
		&B(A\vol(M\setminus\Sigma))^{-1}+\int_{M\setminus\Sigma}||\log h_{\epsilon}||_{L^1(M\setminus\Sigma)}<\Theta[\log h_{\epsilon}](\ol\pa_E\widetilde{\log h_{\epsilon}}),\ol\pa_E\widetilde{\log h_{\epsilon}}>\dvol_g
		\\&\leq-\int_{M\setminus\Sigma}<\sqrt{-1}\Lambda F_{K}^\perp,\widetilde{\log h_{\epsilon}}>\dvol_g.
\end{split}\end{equation}
\par Notice that $l\Theta(lx,ly)\rightarrow(x-y)^{-1}$ for $x>y$ and $l\Theta(lx,ly)\rightarrow+\infty$ for $x\leq y$ increases monotonically as $l\rightarrow+\infty$. It holds
\begin{equation}\begin{split}
		&B(A\vol(M\setminus\Sigma))^{-1}+\int_{M\setminus\Sigma}(\rho[\widetilde{\log h_{\epsilon}}](\ol\pa_E\widetilde{\log h_{\epsilon}}),\ol\pa_E\widetilde{\log h_{\epsilon}})\dvol_g
		\\&\leq-\int_{M\setminus\Sigma}<\sqrt{-1}\Lambda F_{K}^\perp,\widetilde{\log h_{\epsilon}}>\dvol_g,
\end{split}\end{equation}
for $\rho:\mathbb{R}\times\mathbb{R}\rightarrow\mathbb{R}$ with $\rho(x,y)<\frac{1}{x-y}$ whenever $x>y$ and $\epsilon$ small enough. Since $\{|\widetilde{\log h_\epsilon}|\}_{\epsilon>0}$ are uniformly bounded, by taking small $\rho$ it follows that $\{\widetilde{\log h_\epsilon}\}_{\epsilon>0}$ are also uniformly bounded in $L_1^2$. We may assume $\widetilde{\log h_\epsilon}\xrightarrow{\epsilon\rightarrow0} u_\infty$ weakly in $L_{1,\loc}^2$-topology and for large $j$,
\begin{equation}\begin{split}
		\vol(M\setminus(\Sigma\cup M_{j}))\leq(2A)^{-1}.
\end{split}\end{equation}
Therefore we have
\begin{equation}\begin{split}
		\int_{M_{j}}|u_\infty|\dvol_g
		&=\lim\limits_{\epsilon\rightarrow0}\int_{M_{j}}|\widetilde{\log h_\epsilon}|\dvol_g
		\\&=\lim\limits_{\epsilon\rightarrow0}(\int_{M\setminus\Sigma}
		-\int_{M\setminus(\Sigma\cup M_{j})})|\widetilde{\log h_\epsilon}|\dvol_g
		\\&\geq 1-A(2A)^{-1}
		\\&=2^{-1},
\end{split}\end{equation}
which indicates $u_\infty\neq0$. Meanwhile, we also have
\begin{equation}\begin{split}\label{keysemistable}
		&B(A\vol(M\setminus\Sigma))^{-1}+\int_{M\setminus\Sigma}(\rho[u_\infty](\ol\pa_E u_\infty,\ol\pa_E u_\infty)\dvol_g
		\\&\leq-\int_{M\setminus\Sigma}<\sqrt{-1}\Lambda F_{K}^\perp,u_\infty>\dvol_g.
\end{split}\end{equation}
\par Based on $(\ref{keysemistable})$ and \cite[Lemma 5.5]{Si1988}, we know that $u_\infty$ has constant eigenvalues almost everywhere. Let $\lambda_1<...<\lambda_k$ denote the distinct eigenvalues of $u_\infty$ with $2\leq k\leq\rank(E)$. Define a family of smooth nonnegative functions $p_i$ on $\mathbb{R}$ such that $p_i(x)=1$ if $x\leq\lambda_i$ and $p_i(x)=0$ if $x>\lambda_{i+1}$. Set $\pi_i=p_i[u_\infty]$ which means $\pi_i$ has eigenvalues $p_i(\lambda_1),...,p_i(\lambda_{k})$. Using $(\ref{keysemistable})$ again, the regularity result of Uhlenbeck-Yau \cite{UY1986} and following \cite[pages 887-889]{Si1988}, one further proves that each $\pi_i$ represent a nontrivial saturated sub-sheaf $\mathcal{S}_i$ and the negativity of
\begin{equation}\begin{split}\label{finish}
		W
		&\triangleq\lambda_k\deg(E,K)-\sum\limits_{i=1}^{k-1}(\lambda_{i+1}-\lambda_{i})\deg(\mathcal{S}_i,K)
		\\&=\sum\limits_{i=1}^{k-1}(\lambda_{i+1}-\lambda_{i})\rank(\mathcal{S}_i)(\frac{\deg(E,K)}{\rank(E)}-\frac{\deg(\mathcal{S}_i,K)}{\rank(\mathcal{S}_i)}),
\end{split}\end{equation}
which we shall not repeat since there is no essential change. Hence there must exist at least one $\mathcal{S}_i$ violating the semi-stability assumption. In a conclusion, $\{H_\epsilon\}$ is an $\AHE$ structure.
\par Finally for any $\epsilon>0$, we have $\det K=\det H_\epsilon$ and it holds for $j_\epsilon^{'}\geq j_\epsilon$ that
\begin{equation}\begin{split}
		\int_{M_{j_\epsilon}}|\ol\pa_E\log h_{\epsilon,j_\epsilon^{'}}|^2\dvol_g
		&\leq C_\epsilon\int_{M_{j_\epsilon^{'}}}\Theta[\log h_{\epsilon,j_\epsilon^{'}}](\ol\pa_E\log h_{\epsilon,j_\epsilon^{'}}),\ol\pa_E\log h_{\epsilon,j^{'}_\epsilon}>\dvol_g
		\\&\leq-C_\epsilon\int_{M_{j_\epsilon^{'}}}<\sqrt{-1}\Lambda F_{K}^\perp,\log h_{\epsilon,j_\epsilon^{'}}>\dvol_g
		\\&\leq C_\epsilon\sup\limits_{M\setminus\Sigma}|\Lambda F_{K}^\perp|\sup\limits_{M\setminus\Sigma}|\log h_{\epsilon,j_\epsilon^{'}}|\vol(M\setminus\Sigma)
		\\&\leq\epsilon^{-1}C_\epsilon\sup\limits_{M\setminus\Sigma}|\Lambda F_{K}^\perp|^2\vol(M\setminus\Sigma),
\end{split}\end{equation}
where $C_\epsilon$ depends on $\epsilon$. Taking $j^{'}_{\epsilon}\rightarrow\infty$ and $j_{\epsilon}\rightarrow\infty$, we have
\begin{equation}\begin{split}
		\int_{M}|\ol\pa_E\log h_{\epsilon}|^2\dvol_g
		\leq \epsilon^{-1}C_\epsilon\sup\limits_{M\setminus\Sigma}|\Lambda F_{K}^\perp|^2\vol(M\setminus\Sigma).
\end{split}\end{equation}
Hence the proof of Theorem \ref{thm2} is completed.
	\end{proof}
\begin{remark}\label{stableremark}
To obtain the existence of a $\HE$ metric, by $(\ref{keysemistable})$ we have
\begin{equation}\begin{split}
		\int_{M\setminus\Sigma}(\rho[u_\infty](\ol\pa_E u_\infty,\ol\pa_E u_\infty)\dvol_g
		&\leq-\int_{M\setminus\Sigma}<\sqrt{-1}\Lambda F_{K}^\perp,u_\infty>\dvol_g.
\end{split}\end{equation}
Nore this inequality follows from the assumption
\begin{equation}\begin{split}
		||\log h_\epsilon||_{L^1(M\setminus \Sigma)}\xrightarrow{\epsilon\rightarrow0}\infty,
\end{split}\end{equation}
and it implies the nonpositivity of $W$. As a consequence, there is at least one $\mathcal{S}_i$ violating the stability assumption. Hence there must exist a $\HE$ metric compatible with $K$.
\end{remark}
\section{Uniqueness and stabilities of Hermitian-Einstein bundles}
We recall following key assumptions of Simpson. 
\begin{proposition}[Proposition 2.2 in \cite{Si1988}]\label{assumptions}
	Suppose $(M,J,g)$ is a compact K\"{a}hler manifold and $\Sigma\subset M$ is a Zariski closed subset, then $M\setminus\Sigma$ satisfies:
	\begin{enumerate}
		\item It has finite volume.
		\item It admits a nonnegative exhaustion function $\phi$ such that $\Delta_g\phi\in L^\infty(M\setminus\Sigma)$.
		\item If $0\leq f\in C^\infty(M\setminus\Sigma)\cap L^\infty(M\setminus\Sigma)$ satisfies $\Delta_g f\geq-A$ for a constant $A$, then
		\begin{equation}\begin{split}
				\sup\limits_{M\setminus\Sigma}f\leq C_A(1+\int_{M\setminus\Sigma}f\dvol_{g}),
		\end{split}\end{equation}
		where $C_A$ is a constant depending on $A$. Moreoover, if $\Delta_g f\geq0$, then $f$ is a constant.
	\end{enumerate}
\end{proposition}
The second assumption in Proposition \ref{assumptions} ensures
\begin{proposition}[Lemma 5.2 in \cite{Si1988}]\label{L2Stokes}
Suppose a Riemannian manifold $(N,g)$ has an exhaustion function $\phi$ with $\Delta_g\phi\in L^1(N)$, $\eta$ is a $\dim N-1$ form with $|\eta|^2\in L^1(N)$, then
	\begin{equation}\begin{split}
			\lim\limits_{j\rightarrow\infty}\int_{N_j}d\eta=0,
	\end{split}\end{equation}
	for a sequence of exhaustion subsets $\{N_j\}_{j\in\mathbb{N}}$.
\end{proposition}
\begin{remark}\label{nonKahlerL2Stokes}
By Simpson's proof, it is not hard to see that Proposition \ref{L2Stokes} also holds for a non-K\"{a}hler manifold $(N,J,g)$ which satisfies $\sqrt{-1}\pa\ol\pa\omega_g^{\dim_{\mathbb{C}}N-1}\geq0$ and admits an exhaustion function $\phi$ such that $\sqrt{-1}\Lambda\pa\ol\pa\phi\in L^1(N)$.
\end{remark}
To prove the stability of $\HE$ metrics, we first prove two lemmas.
\begin{lemma}\label{keylemma1}
For two Hermitian metrics $H_1,H_2$ on a holomorphic vector bundle $(E,\ol\pa_E)$ and a holomorphic sub-bundle $S\subset E$, we denote by $\pi_{H_1}$ the $H_1$-orthogonal projection onto $S$ and $\ol\pa_{S}$ the induced holomorphic structure on $S$, then we have
\begin{enumerate}
\item $H_1|_{S}^{-1}H_2|_S=\pi_{H_1}\circ H_1^{-1}H_2\circ\pi_{H_1}$, $H_2|_{S}^{-1}H_1|_S=\pi_{H_2}\circ H_2^{-1}H_1\circ\pi_{H_2}$.
\item $\ol\pa_{S}(H_1|_{S}^{-1}H_2|_S)=\pi_{H_1}\circ\ol\pa_{E}(H_1^{-1}H_2)\circ\pi_{H_1}
+\ol\pa_{E}\pi_{H_1}\circ(\id_E-\pi_{H_1})\circ(H_1^{-1}H_2)\circ\pi_{H_1}$.
\end{enumerate}
\end{lemma}
\begin{proof}
For $u,v\in C^\infty(X,S)$, we have
\begin{equation}\begin{split}
H_2|_S(u,v)&=H_1(H_1^{-1}H_2(u),v)
\\&=H_1((\pi_{H_1}\circ H_1^{-1}H_2)(u),v)
\\&=H_1|_S((\pi_{H_1}\circ H_1^{-1}H_2)(u),v),
\end{split}\end{equation}
and this proves (1). Let us write $E=S\oplus S^\perp$ for the $H_1$-orthogonal decomposition and with respect to this decomposition, we have
\begin{equation}\begin{split}
	\ol\pa_E=\left(\begin{array}{ccc}
		\ol\pa_S & -\ol\pa_E\pi_{H_1}\\
		 0 & \ol\pa_{S^\perp}
	 \end{array}\right),\
 H_1^{-1}H_2=\left(\begin{array}{ccc}
 		H_1|_S^{-1}H_2|_S & A\\
 	B & C
 \end{array}\right),
	\end{split}\end{equation}
where $\ol\pa_{S^\perp}$ is the induced holomorphic structure on $S^\perp$. Then we deduce
\begin{equation}\begin{split}
&\ol\pa_E(H_1^{-1}H_2)
\\&=\left(\begin{array}{ccc}
	\ol\pa_S & -\ol\pa_E\pi_{H_1}\\
	0 & \ol\pa_{S^\perp}
\end{array}\right)\left(\begin{array}{ccc}
H_1|_S^{-1}H_2|_S & A\\
B & C
\end{array}\right)
\\&-\left(\begin{array}{ccc}
	H_1|_S^{-1}H_2|_S & A\\
	B & C
\end{array}\right)\left(\begin{array}{ccc}
\ol\pa_S & -\ol\pa_E\pi_{H_1}\\
0 & \ol\pa_{S^\perp}
\end{array}\right)
\\&=\left(\begin{array}{ccc}
	\ol\pa_S\circ H_1|_S^{-1}H_2|_S-\ol\pa_E\pi_{H_1}\circ B & \ol\pa_S\circ A-\ol\pa_E\pi_{H_1}\circ C\\
	\ol\pa_{S^\perp}\circ B & \ol\pa_{S^\perp}\circ C
\end{array}\right)
\\&-\left(\begin{array}{ccc}
H_1|_S^{-1}H_2|_S\circ\ol\pa_S & -H_1|_S^{-1}H_2|_S\circ\ol\pa_E\pi_{H_1}+A\circ\ol\pa_{S^\perp}\\
B\circ\ol\pa_S & -B\circ\ol\pa_E\pi_{H_1}+C\circ\ol\pa_{S^\perp}
\end{array}\right)
\\&=\left(\begin{array}{ccc}
	\ol\pa_{S}(H_1|_{S}^{-1}H_2|_S) & \ol\pa_{(S^\perp)^\ast\otimes S}A
	\\
	\ol\pa_{S^\ast\otimes S^\perp}B & \ol\pa_{S^\perp} C
\end{array}\right)
\\&+\left(\begin{array}{ccc}
-\ol\pa_E\pi_{H_1}\circ B & -\ol\pa_E\pi_{H_1}\circ C+H_1|_S^{-1}H_2|_S\circ\ol\pa_E\pi_{H_1}
	\\
	0 & B\circ\ol\pa_E\pi_{H_1}
\end{array}\right),
\end{split}\end{equation}
and also note $B=(\id_E-\pi_{H_1})\circ(H_1^{-1}H_2)\circ\pi_{H_1}$, this proves $(2)$.
\end{proof}
\begin{lemma}\label{keylemma2}
Let $(E,\ol\pa_E)$ be a holomorphic vector bundle on a K\"{a}hler manifold $(N,J,g)$ admitting an exhaustion function $\phi$ such that $\Delta_g\phi\in L^1(N)$ and $S\subset E$ be a holomorphic sub-bundle. Suppose $H_1,H_2$ are two Hermitian metrics on $E$ compatible with each other and with $|\Lambda F_{H_1}|_{H_1}, |\Lambda F_{H_2}|_{H_2}\in L^1(N)$. Denote by $\pi_{H_1}$, $\pi_{H_2}$ the $H_1$-orthogonal, $H_2$-orthogonal projections onto $S$ respectively, it holds for a sequence of exhaustion subsets $\{N_j\}_{j\in\mathbb{N}}$ that
\begin{equation}\begin{split}
\lim\limits_{j\rightarrow\infty}\int_{N_j}\tr\sqrt{-1}\Lambda F_{H_2|_S}\dvol_g=\deg(S,H_1).
\end{split}\end{equation}
Moreover, $|\ol\pa_E\pi_{H_1}|_{H_1}\in L^2(N)$ if and only if $|\ol\pa_E\pi_{H_2}|_{H_2}\in L^2(N)$.
\end{lemma}
\begin{proof}
We only prove $|\ol\pa_E\pi_{H_1}|_{H_1}\in L^2(N)$ implies $|\ol\pa_E\pi_{H_2}|_{H_2}\in L^2(N)$. Recall
\begin{equation}\begin{split}
		\deg(S,H_1)
		&=\int_{N}\tr\sqrt{-1}\Lambda F_{H_1|_S}\dvol_g
		\\&=\int_{N}\left(\tr(\sqrt{-1}\pi_{H_1}\circ\Lambda F_{H_1})-|\ol\pa_E\pi_{H_1}|^2_{H_1}\right)\dvol_g,
\end{split}\end{equation}
it follows
\begin{equation}\begin{split}\label{finite1}
		\deg(S,H_1)
		&\leq\int_{N}\left(\sqrt{\rank(S)}|\Lambda F_{H_1}|_{H_1}-|\ol\pa_E\pi_{H_1}|^2_{H_1}\right)\dvol_g
		\\&\leq\sqrt{\rank(S)}\int_N|\Lambda F_{H_1}|_{H_1}\dvol_g
		\\&<+\infty,
\end{split}\end{equation}
and also
\begin{equation}\begin{split}\label{finite2}
	    \deg(S,H_1)
	    &\geq-\sqrt{\rank(S)}\int_N\left(|\Lambda F_{H_1}|_{H_1}+|\ol\pa_E\pi_{H_1}|^2_{H_1}\right)\dvol_g
	    \\&>-\infty.
\end{split}\end{equation}
By Lemma \ref{keylemma1} and assumptions, we have
\begin{equation}\begin{split}
		&||\tr((H_1|_S^{-1}K|_S)^{-1}\pa_{K|_S}(H_1|_S^{-1}K|_S))\wedge\omega_g^{\dim_{\mathbb{C}}-1}||_{L^2(N)}
		\\&=||\tr(\ol\pa_{S}(H_1|_S^{-1}K|_S)(H_1|_S^{-1}K|_S)^{-1})\wedge\omega_g^{\dim_{\mathbb{C}}-1}||_{L^2(N)}
		\\&<\infty,
\end{split}\end{equation}
thus Lemma \ref{L2Stokes} implies
\begin{equation}\begin{split}\label{difference1}
		&\lim\limits_{j\rightarrow\infty}\int_{N_j}
		\tr\sqrt{-1}\Lambda\ol\pa_{S}\left((H_1|_S^{-1}K|_S)^{-1}\pa_{K|_S}(H_1|_S^{-1}K|_S)\right)\dvol_g
		\\&=\lim\limits_{j\rightarrow\infty}\int_{N_j}
		\tr\sqrt{-1}\ol\pa_{S}\left((H_1|_S^{-1}K|_S)^{-1}\pa_{K|_S}(H_1|_S^{-1}K|_S)\right)\wedge\frac{\omega_g^{\dim_{\mathbb{C}}N-1}}{(\dim_{\mathbb{C}}N-1)!}
		\\&=\lim\limits_{j\rightarrow\infty}\int_{N_j}
		d\sqrt{-1}\left(\tr((H_1|_S^{-1}K|_S)^{-1}\pa_{K|_S}(H_1|_S^{-1}K|_S))\wedge\frac{\omega_g^{\dim_{\mathbb{C}}N-1}}{(\dim_{\mathbb{C}}N-1)!}\right)
		\\&=0.
\end{split}\end{equation}
Then by $(\ref{finite1})$, $(\ref{finite2})$, $(\ref{difference1})$, we further deduce
\begin{equation}\begin{split}
	&\lim\limits_{j\rightarrow\infty}\int_{N_j}\tr\sqrt{-1}\Lambda F_{H_2|_S}\dvol_g
	\\&=\lim\limits_{j\rightarrow\infty}\int_{N_j}\left(\tr\sqrt{-1}\Lambda F_{H_1|_S}+\tr\sqrt{-1}\ol\pa_S((H_1|_{S}^{-1}H_2|_S)^{-1}\pa_{K|_S}(H_1|_{S}^{-1}H_2|_S))\right)\dvol_g
	\\&=\lim\limits_{j\rightarrow\infty}\int_{N_j}\tr\sqrt{-1}\Lambda F_{H_1|_S}\dvol_g
    \\&+\lim\limits_{j\rightarrow\infty}\int_{N_j}\tr\sqrt{-1}\ol\pa_S((H_1|_{S}^{-1}H_2|_S)^{-1}\pa_{K|_S}(H_1|_{S}^{-1}H_2|_S))\dvol_g
    \\&=\deg(S,H_1),
\end{split}\end{equation}
and therefore
\begin{equation}\begin{split}
		\int_N|\ol\pa_E\pi_{H_2}|_{H_2}^2\dvol_g
		&=\lim\limits_{j\rightarrow\infty}\int_{N_j}|\ol\pa_E\pi_{H_2}|_{H_2}^2\dvol_g
		\\&=\lim\limits_{j\rightarrow\infty}\int_{N_j}\left(\tr(\sqrt{-1}\pi_{H_2}\circ\Lambda F_{H_2})-\tr\sqrt{-1}\Lambda F_{H_2|_S}\right)\dvol_g
		\\&\leq\sqrt{\rank(S)}\int_X|\Lambda F_{H_2}|_{H_2}\dvol_g+\deg(S,H_1)
		\\&<+\infty.
\end{split}\end{equation}
\end{proof}
Now we are in a position to prove the stability of $\HE$ metrics.
\begin{proof}[\textup{\textbf{Proof of Theorem \ref{thm4}}}]
Since $\det K=\det H$ and $|\Lambda F_K|\in L^1(M\setminus\Sigma)$, we have
\begin{equation}\begin{split}\label{thm32}
			\frac{\deg(E,H)}{\rank(E)}
			&=\frac{\int_{M\setminus\Sigma}\tr\sqrt{-1}\Lambda F_{H}\dvol_g}{\rank(E)}
			\\&=\frac{\int_{M\setminus\Sigma}\tr\sqrt{-1}\Lambda F_{K}\dvol_g}{\rank(E)}
			\\&\leq\frac{\int_{M\setminus\Sigma}|\Lambda F_{K}|_K|\id_E|\dvol_g}{\sqrt{\rank(E)}}
			\\&<+\infty.
\end{split}\end{equation}
For any any nontrivial saturated sub-sheaf $\mathcal{S}$ and if $|\ol\pa_E\pi_{H}|_{H}\notin L^2(M\setminus(\Sigma\cup Z(\mathcal{S})))$, we have
\begin{equation}\begin{split}
		\frac{\deg(\mathcal{S},H)}{\rank(\mathcal{S})}
		&=\frac{\int_{M\setminus(\Sigma\cup Z(\mathcal{S}))}\left(\tr(\sqrt{-1}\pi_{H}\circ\Lambda F_{H})-|\ol\pa_E\pi_{H}|^2_{H}\right)\dvol_g}{\rank(\mathcal{S})}
		\\&=\frac{\int_{M\setminus(\Sigma\cup Z(\mathcal{S}))}\left(\tr(\sqrt{-1}\pi_{H}\circ\frac{\tr\Lambda F_H}{\rank(E)})-|\ol\pa_E\pi_{H}|^2_{H}\right)\dvol_g}{\rank(\mathcal{S})}
		\\&=\frac{\int_{M\setminus(\Sigma\cup Z(\mathcal{S}))}\left(\tr(\sqrt{-1}\pi_{H}\circ\frac{\tr\Lambda F_K}{\rank(E)})-|\ol\pa_E\pi_{H}|^2_{H}\right)\dvol_g}{\rank(\mathcal{S})}
		\\&=\frac{\deg(E,K)}{\rank(E)}-\frac{\int_{M\setminus(\Sigma\cup Z(\mathcal{S}))}|\ol\pa_E\pi_{H}|^2_{H}\dvol_g}{\rank(\mathcal{S})}.
\end{split}\end{equation}
and hence $\frac{\deg(\mathcal{S},H)}{\rank(\mathcal{S})}=-\infty<\frac{\deg(E,H)}{\rank(E)}$.
If $|\ol\pa_{E}\pi_{H}|_{H}\in L^2(M\setminus(\Sigma\cup Z(\mathcal{S})))$, $\deg(\mathcal{S},H)$ is finite and using Lemma \ref{keylemma2} with $H_1=H, H_2=K$, we have
\begin{equation}\begin{split}\label{uniqueness1}
		&\lim\limits_{j\rightarrow\infty}\int_{X_j}\tr\sqrt{-1}\Lambda F_{K|_S}\dvol_g
		=\deg(S,H),
\end{split}\end{equation}
for a sequence of exhaustion subsets $\{X_j\}_{j\in\mathbb{N}}$ of $M\setminus(\Sigma\cup Z(\mathcal{S}))$.
Moreover, Lemma \ref{keylemma2} also gives $|\ol\pa_E\pi_K|_K^2\in L^2(M\setminus(\Sigma\cup Z(\mathcal{S})))$ and thus $\tr\sqrt{-1}\Lambda F_{K|_S}$ is absolutely integrable since
\begin{equation}\begin{split}
		&\int_{M\setminus(\Sigma\cup Z(\mathcal{S}))}|\tr\sqrt{-1}\Lambda F_{K|_S}|\dvol_g
		\\&=\int_{M\setminus(\Sigma\cup Z(\mathcal{S}))}|\tr(\sqrt{-1}\pi_{K}\circ\Lambda F_{K})-|\ol\pa_E\pi_{K}|^2_{K}|\dvol_g.
\end{split}\end{equation}
Hence we have
\begin{equation}\begin{split}\label{uniqueness2}
		\lim\limits_{j\rightarrow\infty}\int_{X_j}\tr\sqrt{-1}\Lambda F_{K|_S}\dvol_g
		&=\int_{M\setminus(\Sigma\cup Z(\mathcal{S}))}\tr\sqrt{-1}\Lambda F_{K|_S}\dvol_g
		\\&=\deg(\mathcal{S},K).
\end{split}\end{equation}
By the stability of $(E,\ol\pa_E,K)$, $(\ref{uniqueness1})$ and $(\ref{uniqueness2})$, we deduce
\begin{equation}\begin{split}\label{thm33}
		\frac{\deg(\mathcal{S},H)}{\rank(\mathcal{S})}
		&=\frac{\deg(\mathcal{S},K)}{\rank(\mathcal{S})}
		\\&<\frac{\deg(E,K)}{\rank(E)}
		\\&=\frac{\deg(E,H)}{\rank(E)}.
\end{split}\end{equation}
Therefore, we obtain stability of $(E,\ol\pa_E,H)$ from stability of $(E,\ol\pa_E,K)$.
\end{proof}
Based on the stability of $\HE$ metrics, we are able to prove the uniqueness of $\HE$ metrics.
\begin{proof}[\textup{\textbf{An alternative proof of Theorem \ref{thm3}}}]
Firstly, we compute
\begin{equation}\begin{split}\label{difference2}
\Delta_g\tr(H_1^{-1}H_2)
&=2\sqrt{-1}\Lambda\pa\ol\pa\tr(H_1^{-1}H_2)
\\&=<\sqrt{-1}\Lambda(F_{H_1}-F_{H_2}),H_1^{-1}H_2>_{H_1}
\\&+\sqrt{-1}\Lambda<H_2^{-1}H_1\pa_{H_1}(H_1^{-1}H_2),\ol\pa_{H_1}(H_1^{-1}H_2)>_{H_1}
\\&=|(H_2^{-1}H_1)^{-\frac{1}{2}}\pa_{H_1}(H_1^{-1}H_2)|^2_{H_1}.
\end{split}\end{equation}
By Proposition \ref{assumptions}, we get $\pa_{H_1}(H_1^{-1}H_2)=0$ and $\ol\pa_{E}(H_1^{-1}H_2)=0$ since $H_1^{-1}H_2$ is self-adjoint with respect to both $H_1$, $H_2$. Let $E=\oplus_{i=1}^{m}E_i$ denote the eigendecomposition of $H_1^{-1}H_2$. Since $H_1^{-1}H_2$ is $D_{H_1}$-parallel, the eigenvalues $\{c_i\}_{i=1}^m$ are constants and therefore 
\begin{equation}\begin{split}
		H_2|_{E_i}=c_iH_1|_{E_i},\ i=1,...,m.
\end{split}\end{equation}
For $v_i\in C^\infty(M\setminus\Sigma,E_i), v_j\in C^\infty(M\setminus\Sigma,E_j)$, we have
\begin{equation}\begin{split}
H_1^{-1}H_2(\ol\pa_Ev_i)=\ol\pa_E((H_1^{-1}H_2)(v_i))=c_i\ol\pa_Ev_i,
\end{split}\end{equation}
\begin{equation}\begin{split}
		H_2(v_i,v_j)=H_1(H_1^{-1}H_2(v_i),v_j)=c_iH_1(v_i,v_j),
\end{split}\end{equation}
\begin{equation}\begin{split}
		 H_2(v_i,v_j)=H_1(v_i,H_1^{-1}H_2(v_j))=c_jH_1(v_i,v_j).
\end{split}\end{equation}
Thus the decomposition splits in $H_1$-orthogonal, $H_2$-orthogonal and holomorphic senses.
\par It suffices to show $m=1$. Otherwise, each $E_i$ is a nontrivial holomorphic sub-bundle and
\begin{equation}\begin{split}
\frac{\deg(E,H_1)}{\rank(E)}
&=\frac{\int_{M\setminus\Sigma}\tr\sqrt{-1}\Lambda F_{H_1}\dvol_g}{\rank(E)}
\\&=\sum\limits_{i=1}^m\frac{\int_{M\setminus\Sigma}\tr\sqrt{-1}\Lambda F_{H_1|_{E_i}}\dvol_g}{\rank(E)}
\\&=\sum\limits_{i=1}^m\frac{\int_{M\setminus\Sigma}\tr\sqrt{-1}\Lambda F_{H_1|_{E_i}}\dvol_g}{\rank(E_i)}\frac{\rank(E_i)}{\rank(E)}
\\&=\sum\limits_{i=1}^m\frac{\deg(E_i,H_1)}{\rank(E_i)}\frac{\rank(E_i)}{\rank(E)}.
\end{split}\end{equation}
On the other hand, recall that we have proved the stability of $(E,\ol\pa_E,H_1)$ in Theorem \ref{thm4}, so it leads to the following contradiction
\begin{equation}\begin{split}
		\frac{\deg(E,H_1)}{\rank(E)}
		&<\frac{\deg(E,H_1)}{\rank(E)}\sum\limits_{i=1}^m\frac{\rank(E_i)}{\rank(E)}
		=\frac{\deg(E,H_1)}{\rank(E)},
\end{split}\end{equation}
and therefore $H_1=H_2$.
\end{proof}
\begin{remark}
Since the proof only involves stabilities of $(E,\ol\pa_E,H_1)$ and $(E,\ol\pa_E,K)$, the assumption $|\ol\pa_E(K^{-1}H_2)|\in L^2(M\setminus\Sigma)$ is superfluous.
\end{remark}
Next we prove the stability (resp. semi-stability) of a Hermitian holomorphic vector bundle which admits $\HE$ metrics (resp. $\AHE$ structures).
\begin{proof}[\textup{\textbf{Proof of Theorem \ref{thm5}}}]
For (1), once $H$ is a $\HE$ metric, it is well-known that $(E,\ol\pa_E,H)$ is a direct sum of some holomorphic vector bundles $(E_i,\ol\pa_E|_{E_i},H|_{E_i})$ and each of them is stable, see \cite{Ko1982,Lu1983} for example. Since $(E,\ol\pa_E)$ is assumed to be indecomposable, we further find the stability of $(E,\ol\pa_E,H)$, from which and a similar argument in our proof Theorem \ref{thm4} (only need to exchange the roles of $K$ and $H$), we can conclude the the stability of $(E,\ol\pa_E,K)$.
\par Next we prove (2). Note in compact case, the semi-stability of a holomorphic vector bundle admitting an $\AHE$ structure follows from a vanishing theorem \cite{Ko1987}. It requires additional work to generalize the argument to noncompact context. The proof below bypasses the use of any vanishing theorem which seems more natural and directly. Let $\{H_\epsilon\}_{\epsilon>0}$ be an $\AHE$ structure compatible with $K$. Following \cite{WZ2021}, we have for any saturated nontrivial sub-sheaf $\mathcal{S}$ that
	\begin{equation}\begin{split}\label{thm31}
			\frac{\deg(\mathcal{S},H_\epsilon)}{\rank(\mathcal{S})}
			&=\frac{\int_{M\setminus(\Sigma\cup Z(\mathcal{S}))}\tr(\pi_{H_\epsilon}\circ\sqrt{-1}\Lambda F_{H_\epsilon})-|\ol\pa_E\pi_{H_\epsilon}|^2_{H_\epsilon}\dvol_g}{\rank(\mathcal{S})}
			\\&\leq\frac{\int_{M\setminus(\Sigma\cup Z(\mathcal{S}))}\tr(\pi_{H_\epsilon}\circ\sqrt{-1}\Lambda F_{H_\epsilon})\dvol_g}{\rank(\mathcal{S})}
			\\&=\frac{\int_{M\setminus(\Sigma\cup Z(\mathcal{S}))}\tr(\pi_{H_\epsilon}\circ\frac{\tr\sqrt{-1}\Lambda F_{H_\epsilon}}{\rank(E)})+\tr(\pi_{H_\epsilon}\circ\sqrt{-1}\Lambda F_{H_\epsilon}^\perp)\dvol_g}{\rank(\mathcal{S})}
			\\&\leq\frac{\deg(E,K)}{\rank(E)}+\frac{\int_{M\setminus(\Sigma\cup Z(\mathcal{S}))}|\pi_{H_\epsilon}|_{H_\epsilon}|\Lambda F_{H_\epsilon}^\perp|_{H_\epsilon}\dvol_g}{\rank(\mathcal{S})}
			\\&\leq\frac{\deg(E,K)}{\rank(E)}+\frac{\epsilon\vol(M)}{\sqrt{\rank(\mathcal{S})}}.
	\end{split}\end{equation}
If $|\ol\pa_E\pi_K|_K\notin L^2(M\setminus(\Sigma\cup Z(\mathcal{S})))$, we have
	\begin{equation}\begin{split}
			\frac{\deg(\mathcal{S},K)}{\rank(\mathcal{S})}
			&\leq\frac{\int_{M\setminus(\Sigma\cup Z(\mathcal{S}))}\left(\sqrt{\rank(\mathcal{S})}|\Lambda F_K|-|\ol\pa_E\pi_K|_K^2\right)\dvol_g}{\rank(\mathcal{S})}
			\\&=-\infty
			\\&<\frac{\deg(E,K)}{\rank(E)}.
	\end{split}\end{equation}
If $|\ol\pa_{\theta}\pi_K|_K\in L^2(M\setminus(\Sigma\cup Z(\mathcal{S})))$, we know $\deg(\mathcal{S},K)$ is finite. Note that
	\begin{equation}\begin{split}
		\int_{M\setminus\Sigma}|\Lambda F_{H_\epsilon}|_{H_\epsilon}\dvol_g
		&=\int_{M\setminus\Sigma}(|\Lambda F_{H_\epsilon}^\perp|_{H_\epsilon}+|\frac{\tr\Lambda F_{K}}{\rank(E)}\id_E|_{H_\epsilon})\dvol_g,
\end{split}\end{equation}
therefore $|\Lambda F_{H_\epsilon}|_{H_\epsilon}\in L^1(M\setminus\Sigma)$ since $\{H_\epsilon\}$ is an $\AHE$ structure. Using Lemma \ref{keylemma2} with $H_1=K, H_2=H_\epsilon$ and that in the proof of Theorem \ref{thm4} as before, we can conclude
\begin{equation}\begin{split}
		\deg(\mathcal{S},H_{\epsilon})
		&=\lim\limits_{j\rightarrow\infty}\int_{X_j}\tr\sqrt{-1}\Lambda F_{H_{\epsilon}|_S}\dvol_g
		=\deg(S,K),
\end{split}\end{equation}
for a sequence of exhaustion subsets $\{X_j\}_{j\in\mathbb{N}}$ of $M\setminus(\Sigma\cup Z(\mathcal{S}))$.
Hence we have
	\begin{equation}\begin{split}\label{thm35}
			\frac{\deg(\mathcal{S},K)}{\rank(\mathcal{S})}
			&=\frac{\deg(\mathcal{S},H_\epsilon)}{\rank(\mathcal{S})}
			\\&\leq\frac{\deg(E,K)}{\rank(E)}+\frac{\epsilon\vol(M)}{\sqrt{\rank(\mathcal{S})}}
			\\&\xrightarrow{\epsilon\rightarrow0}\frac{\deg(E,K)}{\rank(E)}.
	\end{split}\end{equation}
The proof is completed.
\end{proof}
Our proof also works for $\mathbb{C}$, $\mathbb{C}\times M$ for a compact K\"{a}hler manifold, and non-K\"{a}hler context.
\begin{theorem}\label{C1}
	Let $(E,\ol\pa_E,K)$ be a stable Hermitian holomorphic vector bundle on $\mathbb{C}$ with $|F_K|\in L^1(\mathbb{C})$ and $H_1, H_2$ be two $\HE$ metrics compatible with $K$, then we have $H_1=H_2$.
\end{theorem}
\begin{proof}
Tracking the proof, it is not hard to conclude Theorems \ref{thm3}, \ref{thm4} for the case where the base space is $\mathbb{C}$. In fact, instead of $(\ref{difference1})$ in Lemma \ref{keylemma2}, we note
	\begin{equation}\begin{split}
			&|\lim\limits_{R\rightarrow\infty}\int_{B_{2R}}
			\eta_R^2\tr\sqrt{-1}\ol\pa_{S}\left((H_1|_S^{-1}K|_S)^{-1}\pa_{K|_S}(H_1|_S^{-1}K|_S)\right)|
			\\&=|\lim\limits_{R\rightarrow\infty}\int_{B_{2R}}
			\eta_R^2d\sqrt{-1}\left(\tr((H_1|_S^{-1}K|_S)^{-1}\pa_{K|_S}(H_1|_S^{-1}K|_S))\right)|
			\\&=2|\lim\limits_{R\rightarrow\infty}\int_{B_{2R}\setminus B_R}
			\eta_Rd\eta_R\wedge\sqrt{-1}\left(\tr((H_1|_S^{-1}K|_S)^{-1}\pa_{K|_S}(H_1|_S^{-1}K|_S))\right)|
			\\&\leq2\lim\limits_{R\rightarrow\infty}\vol(B_{2R}\setminus B_R)|d\eta_R|^2
			||\tr((H_1|_S^{-1}K|_S)^{-1}\pa_{K|_S}(H_1|_S^{-1}K|_S))||_{L^2(B_{2R}\setminus B_R)}
			\\&\leq2C^2\lim\limits_{R\rightarrow\infty}R^{-2}\vol(B_{2R}\setminus B_R)
			||\tr((H_1|_S^{-1}K|_S)^{-1}\pa_{K|_S}(H_1|_S^{-1}K|_S))||_{L^2(B_{2R}\setminus B_R)}
			\\&=0,
	\end{split}\end{equation}
where $\{\eta_R\in C_0^\infty(B_{2R})\}_{R\in\mathbb{N}}$ is a sequence of cut-off functions such that $0\leq\eta_R\leq1$, $\eta_R=1$ on $B_R$, and $|d\eta_R|\leq CR^{-1}$ for constant $C$. Therefore it holds
\begin{equation}\begin{split}
		\lim\limits_{R\rightarrow\infty}\int_{B_{2R}}
		\eta_R^2\tr\sqrt{-1}F_{H_2|_S}=\deg(S,H_1).
\end{split}\end{equation}
Moreover, we also need to use the Liouville theorem for bounded sub-harmonic functions.
\end{proof}
\begin{theorem}\label{CM1}
	Let $(E,\ol\pa_E,K)$ be a stable Hermitian holomorphic vector bundle on $\mathbb{C}\times M$ for a compact K\"{a}hler manifold $M$ with $|F_K|\in L^1(\mathbb{C}\times M)$ and $H_1, H_2$ be two $\HE$ metrics compatible with $K$, then we have $H_1=H_2$.
\end{theorem}
\begin{proof}
It suffices to modify the proof of Theorem \ref{C1} by replacing the sequence of cut-off functions $\{\eta_r\in C_0^\infty(B_{2R})\}_{R\in\mathbb{N}}$ by $\{\eta_{R,M}\in C_0^\infty(B_{2R}\times M)\}_{R\in\mathbb{N}}$, where $\eta_{R,M}(z,m)\triangleq\eta_R(z)$ and integrating over $B_{2R}\times M$. Moreover, we also note that the Liouville theorem for bounded sub-harmonic functions also holds on $\mathbb{C}\times M$.
	\end{proof}
\begin{theorem}\label{C2}
Let $(E,\ol\pa_E,K)$ be an indecomposable Hermitian holomorphic vector bundle on $\mathbb{C}$ with $|F_K|\leq C(1+|z|^2)^{-p}$ for some $p\in(1,2]$ and positive constant $C$, then 
\begin{center}
	$(E,\ol\pa_E,K)$ is stable $\Leftrightarrow$ it admits a unique $\HE$ metric compatible with $K$.
\end{center}
\end{theorem}
\begin{proof}
Similar as Theorem \ref{C1}, Theorem \ref{thm5}-(1) also holds for the case where the base space is $\mathbb{C}$. By \cite[Theorem 1.2]{Mo2020}, it suffices to verify that $\mathbb{C}$ must satisfy \cite[Assumption 2.1]{Mo2020} with respect some suitable underlying metrics, readers may consult \cite[Corollary 3.2]{WZ2023}.
\end{proof}
\begin{theorem}\label{thm71}
	Let $(M,J,g)$ be a compact Gauduchon manifold and $\Sigma\subset M$ be a smooth divisor. Suppose $(E,\ol\pa_E,K)$ is a stable Hermitian holomorphic vector bundle on $M\setminus\Sigma$ with $|\Lambda F_K|\in L^\infty(M\setminus\Sigma)$ and $H_1, H_2$ are two $\HE$ metrics compatible with $K$, then we have $H_1=H_2$.
\end{theorem}
\begin{theorem}\label{thm72}
	Let $(M,J,g)$ be a compact Gauduchon manifold and $\Sigma\subset M$ be a smooth divisor. Suppose $(E,\ol\pa_E,K)$ is an indecomposable Hermitian holomorphic vector bundle on $M\setminus\Sigma$ with $|\Lambda F_K|\in L^\infty(M\setminus\Sigma)$, then 
	\begin{center}
		$(E,\ol\pa_E,K)$ is stable $\Leftrightarrow$ it admits a unique $\HE$ metric compatible with $K$.
		\end{center}
\end{theorem}
\begin{theorem}\label{thm73}
	Let $(M,J,g)$ be a compact Gauduchon manifold and $\Sigma\subset M$ be a smooth divisor. Suppose $(E,\ol\pa_E,K)$ is a Hermitian holomorphic vector bundle on $M\setminus\Sigma$ with $|\Lambda F_K|\in L^\infty(M\setminus\Sigma)$, then 
	\begin{center}
		$(E,\ol\pa_E,K)$ is semi-stable $\Leftrightarrow$ it admits an $\AHE$ structure compatible with $K$. 
	\end{center}
\end{theorem}
In Theorems \ref{thm71}, \ref{thm72}, \ref{thm73}, $M\setminus\Sigma$ clearly satisfies assumptions (1), (2) in Proposition \ref{assumptions} as long as $\Delta_g$ is replaced by $\sqrt{-1}\Lambda\pa\ol\pa$. The existence part follows from Theorems \ref{thm1}, \ref{thm2}. As for the remaining part, the previous discussion in K\"{a}hler context also works if one uses the non-K\"{a}hler analogue of Proposition \ref{L2Stokes} (see Remark \ref{nonKahlerL2Stokes}) and the following variant of $(\ref{difference1})$:
\begin{equation}\begin{split}
		&\lim\limits_{j\rightarrow\infty}\int_{N_j}
		\tr\sqrt{-1}\Lambda\ol\pa_{S}\left((H_1|_S^{-1}K|_S)^{-1}\pa_{K|_S}(H_1|_S^{-1}K|_S)\right)\dvol_g
		\\&=\lim\limits_{j\rightarrow\infty}\int_{N_j}
		d\sqrt{-1}\left(\tr((H_1|_S^{-1}K|_S)^{-1}\pa_{K|_S}(H_1|_S^{-1}K|_S))\wedge\frac{\omega_g^{\dim_{\mathbb{C}}N-1}}{(\dim_{\mathbb{C}}N-1)!}\right)
		\\&+\lim\limits_{j\rightarrow\infty}\int_{N_j}
		\sqrt{-1}\left(\tr(\pa_{K|_S}\log(H_1|_S^{-1}K|_S))\wedge\frac{\ol\pa\omega_g^{\dim_{\mathbb{C}}N-1}}{(\dim_{\mathbb{C}}N-1)!}\right)
		\\&=\lim\limits_{j\rightarrow\infty}\int_{N_j}
		d\sqrt{-1}\left(\tr((H_1|_S^{-1}K|_S)^{-1}\pa_{K|_S}(H_1|_S^{-1}K|_S))\wedge\frac{\omega_g^{\dim_{\mathbb{C}}N-1}}{(\dim_{\mathbb{C}}N-1)!}\right)
		\\&+\lim\limits_{j\rightarrow\infty}\int_{N_j}
		d\sqrt{-1}\left(\tr\log(H_1|_S^{-1}K|_S)\wedge\frac{\ol\pa\omega_g^{\dim_{\mathbb{C}}N-1}}{(\dim_{\mathbb{C}}N-1)!}\right)
		\\&=0.
\end{split}\end{equation}
\section*{Acknowledgement}
The first author would like to thank Doctor Chao Li for several useful discussion, to Professor Takuro Mochizuki for insightful discussions and comments.

\end{document}